\documentclass[10pt,reqno]{amsart}
\numberwithin{equation}{section}

\usepackage{amssymb}
\usepackage{enumerate, xspace}
\usepackage{marvosym} 
\usepackage[sans]{dsfont} 
\usepackage{bbm}
\usepackage{changebar}
\usepackage[colorlinks]{hyperref}

\usepackage[usenames,dvipsnames,svgnames,table]{xcolor}

\usepackage{ulem}
\usepackage{setspace}

\newtheorem{thm}{Theorem}[section]
\newtheorem{lem}[thm]{Lemma}

\newtheorem{prop}[thm]{Proposition}

\newtheorem{definition}[thm]{Definition}

\newtheorem{rem}[thm]{Remark}

\newcommand\bZ{{\mathbb Z}}

\newcommand\ve{\varepsilon}

\newcommand\LL{{\mathbb L}}
\newcommand\EE{{\mathbb E}}
\newcommand\PP{{\mathbb P}}

\newcommand\RR{{\mathbb R}}

\newcommand\ZZ{{\mathbb Z}}

\newcommand{\mc}[1]{{\mathcal #1}}

\newcommand{\bb}[1]{{\mathbb #1}}

\begin{document}

\title[Fractional Fick's law]{Fractional Fick's law for the boundary driven exclusion process with long jumps}
\author{C.Bernardin}
\address{Universit\'e C\^ote d'Azur, CNRS, LJAD\\
Parc Valrose\\
06108 NICE Cedex 02, France}
\email{{\tt cbernard@unice.fr}}

\author{B. Jim\'enez Oviedo}
\address{Universit\'e C\^ote d'Azur, CNRS, LJAD\\
Parc Valrose\\
06108 NICE Cedex 02, France}
\email{{\tt byron@unice.fr}}
\thanks{}

\date{\today.}
\begin{abstract} A fractional Fick's law and fractional hydrostatics for the one dimensional  exclusion process with long jumps in contact with infinite reservoirs at different densities on  the left and on the right are derived.
\end{abstract}

\keywords{Fick's law, Exclusion Process with long jumps, Fractional diffusion, Superdiffusion} 
\subjclass{60K35, 82C22, 35R11, 60G51}

\maketitle

\section{Introduction}

The exclusion process is known as the ``Ising model" of non-equilibrium statistical mechanics and since its introduction in the 70's in biophysics by MacDonald et al. (\cite{MacDo1,MacDo2}) and in probability by Spitzer \cite{Spitzer}, a lot of papers in the mathematical physics literature focused on it because it captures the main features of more realistic diffusive systems driven out of equilibrium (\cite{Li}, \cite{Liggett2}, \cite{Spohn}). The exclusion process is an interacting particle system consisting of a collection of continuous-time dependent random walks moving on the lattice $\ZZ$:  A particle at $x$ waits an exponential time and then chooses to jump to $x+y$ with probability  $p(y)$.  If, however, $x+y$ is already occupied, the jump is suppressed and the clock is reset. In this paper we are interested in the case where $p(\cdot)$ has a long tail, proportional to $|\cdot|^{-(1+\gamma)}$ for $\gamma>1$. Curiously it is only very recently that the investigation of the exclusion process with long jumps started (\cite{BGS,Jara1, Jara2, GJ2,Seth, SNU}).    

Our motivation for this study is threefold. First, due to the intense activity developed around the exclusion process since its introduction almost fifty years ago, it is very natural to investigate on the differences and the similarities between the {\textit{finite jumps}} exclusion process and the {\textit{long jumps}} exclusion process. Our second motivation is related to the field of anomalous diffusion in one dimensional chains of oscillators (\cite{D, LLP, Sp2}). Recent studies suggest that the macroscopic behavior of some chains of oscillators (with short range interactions) displaying anomalous diffusion should be similar to the macroscopic behavior of the symmetric exclusion process with long jumps. In order to motivate this claim, let us observe that the equilibrium fluctuations of a harmonic chain with energy-momentum conservative noise and of the long jumps exclusion process with exponent $\gamma=3/2$ are the same (\cite{Jara1, Jara2, BGJ, BGJSS,JKO2}). See also Remark \ref{rem:rem1} of this paper for a second example. These similarities can be roughly understood by the fact that in $1d$ chains of oscillators, the energy carriers, the phonons, do not behave like interacting Brownian particles but like interacting Levy walks (\cite{DSS,DS} and \cite{DKZ} for a review on Levy walks). Therefore, we believe that the symmetric exclusion process with long jumps could play the role of a simple effective model to investigate properties of superdiffusive chains of oscillators. Our third motivation, which is related to the second but has also its own interest, is to develop a {\textit{macroscopic fluctuation theory}} for superdiffusive systems (e.g. exclusion process with long jumps) as it has been done during the last decade by Bertini et al. (\cite{BDGGL}) for diffusive systems. The key idea behind the macroscopic fluctuation theory is that the non-equilibrium free energy of a particular given system  depends only on its macroscopic behavior and not on its microscopic details. Therefore, two models macroscopically identical  shall have the same non-equilibrium free energie. As explained above our hope is that some superdiffusive chains of oscillators and exclusion processes with long jumps have the same macroscopic behavior and hence the same non-equilibrium free energy.  

In this paper we consider the symmetric exclusion process with long jumps in contact with two reservoirs with different densities at the boundaries. We show that in the non-equilibrium stationary state the average density current scales with the length $N$ of the system as $N^{-\delta}$, $0<\delta<1$. We also show that the stationary density profile is described by the stationary solution of a fractional diffusion equation with Dirichlet boundary conditions. Observe that in a diffusive regime, $\delta=1$ and that the stationary profile is the stationary solution of a usual diffusion equation with Dirichlet boundary conditions. Similar conclusions to ours, as well as extensions to the asymmetric case, have been obtained in a non-rigorous physics paper by J. Szavits-Nossan and K. Uzelac (\cite{SNU}). As a final remark of this introduction let us observe that in our paper, as well as in \cite{SNU}, the reservoirs are described by infinite reservoirs. This has the advantage to avoid a truncation of the long range transition probability $p(\cdot)$. However other reservoirs descriptions are possible but we conjecture that they could have a quantitative effect on the form of the stationary profile. Indeed, since the fractional Laplacian is a non-local operator, the fractional Laplacian with Dirichlet boundary conditions can be interpreted in several ways giving rise to different stationary solutions. The (microscopic) description used for the reservoirs fix the (macroscopic) interpretation of the fractional Laplacian with Dirichlet boundary conditions. In our case it is the so called ``restricted fractional Laplacian" which appears. This sensitivity to the form of the reservoirs is due to the presence of long jumps and does not appear for the exclusion process with short jumps. This sensitivity has also been observed in models of (non interacting) Levy walks and in the context of $1d$ superdiffusive chains of oscillators (\cite{LP}).  
   
The paper is organized as follows. In Section \ref{sec:model}  we describe precisely the model studied and the results obtained. In Section \ref{sec:wsfhe} we recall basic facts on the fractional Laplacian and explain what we mean by stationary solution of a fractional diffusion equation with Dirichlet boundary conditions. Section \ref{sec:Proofs} is devoted to the proofs of the results with some technical lemmata postponed to the Appendix.  
   
\section{Model and Results}
\label{sec:model}
We consider a symmetric long jumps exclusion process on $\Lambda_N=\{1, \ldots, N-1\}$, $N \ge 2$, in contact with two reservoirs at density $\alpha \in (0,1)$ on the left and density $\beta \in(0,1)$ on the right. Let $p(\cdot)$ be a probability function on $\ZZ$ which takes the form
$$p(z)= \cfrac{c_{\gamma}}{|z|^{1+\gamma}}, \quad |z| \ge 1,\quad p(0)=0,$$
where $2>\gamma >1$ and $c_\gamma>0$ is a normalization factor. If $\gamma\ge 2$ the boundary driven long jumps symmetric exclusion process has {\textit{mutatis mutandis}} the same behavior as the usual  boundary driven finite jumps exclusion process {\footnote{For $\gamma=2$ the diffusive scaling has to be replaced by a diffusive scaling with some logarithmic corrections but the system behaves macroscopically in a diffusive way (\cite{Jara1}, Appendix A). }}. The configuration space of the process is $\Omega_N = \{ 0, 1\}^{\Lambda_N}$ and a typical configuration $\eta$ is denoted as a sequence $(\eta_z)_z$ indexed by $z\in \Lambda_N$. The generator of the boundary driven symmetric long jumps exclusion process $\{ \eta (t) \; ; \; t \ge 0\}$ is defined by
\begin{equation}
L_N = L^0_N  + L_N^{r} + L_N^{\ell} 
\end{equation}
where for any $f:\Omega_N \to \RR$ 
\begin{equation}
\begin{split}
(L^0_N f)(\eta) &=\sum_{x,y \in \Lambda_N} p(x-y) \eta_x (1-\eta_y) [ f(\eta^{xy}) -f(\eta)] \\
&= \cfrac{1}{2} \, \sum_{x,y \in \Lambda_N} p(x-y) [ f(\eta^{xy}) -f(\eta)],\\
(L_N^{r} f)(\eta)&= \sum_{x \in \Lambda_N, y \ge N} p(x-y) [ \eta_x (1-\beta) + (1-\eta_x) \beta]  [f(\eta^x) - f(\eta)],\\
(L_N^{\ell} f)(\eta) &= \sum_{x \in \Lambda_N, y \le 0} p(x-y) [ \eta_x (1-\alpha) + (1-\eta_x) \alpha]  [f(\eta^x) - f(\eta)].
\end{split}
\end{equation}
Here the configurations $\eta^x$ and $\eta^{xy}$ are defined by
\begin{equation*}
(\eta^{xy})_z = 
\begin{cases}
\eta_z, \; z \ne x,y,\\
\eta_y, \; z=x,\\
\eta_x, \; z=y
\end{cases}
, \quad (\eta^x)_z= 
\begin{cases}
\eta_z, \; z \ne x,\\
1-\eta_x, \; z=x.
\end{cases}
\end{equation*}
Sometimes it will be useful to consider a configuration $\eta \in \Omega_N$ as a configuration on $\{0,1,\alpha,\beta\}^{\ZZ}$ by extending $\eta$ by setting $\eta_x =\alpha$ for $x\le 0$ and $\eta_x=\beta$ for $x\ge N$. Observe that the reservoirs add and remove particles on all the sites of the lattice $\Lambda_N$, and not only at the boundaries, but with rates which decrease as the distance from the corresponding reservoir increases. The same kind of reservoirs is used in \cite{SNU}. 

The bulk dynamics (i.e. without the presence of the reservoirs) conserves the number of particles. Let $W_x$, $x=1, \ldots, N$, be defined by
\begin{equation}
\begin{split}
&W_x= \sum_{y \le x-1} \sum_{z\ge x} p(z-y) [\eta_y -\eta_z] \, + (\beta- \alpha) \sum_{\substack{y\le 0\\ z \ge N}} p(z-y).
\end{split}
\end{equation}
In this formula, as explained above, we adopted the convention $\eta_z= \alpha$ for $z \le 0$ and $\eta_z =\beta$ for $z\ge N$. It can be checked that since $\gamma>1$, these quantities are well defined. Observe that the quantity $W_x$ is equal to 
\begin{equation*}
\begin{split}
W_x&=\sum_{1 \le y\le x-1 <z\le N-1} \; p(z-y) [\eta_y (1-\eta_z) -\eta_z (1-\eta_y) ] \\
&+\sum_{x \le z \le N-1} \sum_{y\le 0} p (z-y) (\alpha -\eta_z) -\sum_{1 \le y\le x-1} \sum_{z \ge N} p(z-y) (\beta-\eta_y).
\end{split}
\end{equation*}
It corresponds to the rate of particles jumping in the bulk by crossing $x-1/2$ from the left to the right minus the rate of particles jumping in the bulk by crossing $x-1/2$ from the right to the left (first sum) plus the rate of particles coming from the left reservoir by crossing $x-1/2$ (second sum) minus the rate of particles coming from the right reservoir by crossing $x-1/2$ (third sum). Then for any $x\in \Lambda_N$ we have the following microscopic continuity equation
\begin{equation}
\label{eq:continuity}
L_N \eta_x = -\nabla W_x := -(W_{x+1} -W_{x}). 
\end{equation}

\begin{rem}
Observe that for each $x \in \Lambda_N$, the current due to the bulk dynamics $\sum_{y \le x-1} \sum_{z\ge x} p(z-y) [\eta_y -\eta_z]$ can be written as a sum of discrete gradients $\sum_k \alpha_x (k) ( \eta_{k+1} -\eta_k)$. Therefore, it belongs to the class of so-called ``gradient models" (see \cite{KL} for more explanations). However the function $\alpha_x (\cdot)$ is not exponentially localized around $x$ so that the model is quite different from a standard ``gradient model". 
\end{rem}

Let us denote by $\mu_N$ the unique invariant measure of $\{ \eta (t) \; ;\; t \ge 0\}$. If $\alpha=\beta=\rho$ then $\mu_N$ is equal to the Bernoulli product measure with density $\rho$. It is denoted by $\nu_{\rho}$. The expectation of a function $f$ with respect to $\mu_N$ (resp. $\nu_{\rho}$) is denoted by $\langle f \rangle_{N}$ (resp. $\langle f \rangle_\rho$) or ${\mu_N} (f)$ (resp. $\nu_\rho (f)$). For any $\rho \in (0,1)$ the density of $\mu_N$ with respect to $\nu_{\rho}$ is denoted by $f_{N, \rho}$.

Let ${\bar \rho}$ be the unique weak solution (see Section \ref{sec:wsfhe} for a precise definition ) of the stationary fractional heat equation with Dirichlet boundary conditions
 \begin{equation}
 \label{eq:fhed}
 \begin{cases}
 &(-\Delta)^{\gamma/2} \, {\bar \rho} (q) =0, \quad q \in (0,1),\\
 &{\bar \rho} (0)=\alpha,\\
 &{\bar \rho}(1)=\beta.
 \end{cases}
 \end{equation}
We have that (see \cite{BB})
\begin{equation}
\label{eq:explicit}
\forall q \in (0,1), \quad 
{\bar \rho}(q) = \int_{\vert y-\frac{1}{2} \vert>\frac{1}{2}}g(y)\; P_{\tfrac{1}{2}} \big(q-\tfrac{1}{2},y-\tfrac{1}{2}\big)\; dy,
\end{equation}
where the function $g$ is given by
\[ g(y) =
  \begin{cases}
    \alpha \quad \text{if} \quad y<0, \\
    \beta \quad \text{if} \quad y>1,\\
    0 \quad \text{otherwise,}
  \end{cases}
\] 
and the Poisson kernel $P_r (\cdot -\theta, \cdot - \theta)$, $r>0, \theta \in \RR$, is defined by 
$$P_{r}(q-\theta,y-\theta) = C_{\gamma} \left[ \dfrac{r^{2}-(q-\theta)^{2}}{(y-\theta)^{2}-r^{2}} \right]^{\frac{\gamma}{2}}\vert q-y \vert^{-1},$$ for $\vert q-\theta\vert<r,\vert y-\theta\vert>r$ and equal to $0$ elsewhere. Here $C_\gamma$ is a normalization constant equal to $C_\gamma= \Gamma(1/2)  \pi^{-3/2} \sin(\pi \gamma /2)$.
It can be shown that the function ${\bar \rho}$ is smooth in the bulk but only $\gamma/2$-H\"older at the boundaries.

Our first result is the hydrostatic behavior for the boundary driven exclusion process with long jumps, stated in the following theorem.
\begin{thm}
\label{thm:weak-profile}
Let $\gamma \in (1,2)$. For any continuous function $H:[0,1] \to \RR$ we have that
\begin{equation*}
\lim_{N \to \infty} \cfrac{1}{N-1} \sum_{z=1}^{N-1} H(\tfrac{z}{N}) \eta_z = \int_0^1 H(q) {\bar \rho }(q) dq
\end{equation*}
in probability under $\mu_N$.
\end{thm}

\begin{rem}
\label{rem:rem1}
In \cite{BKO} and \cite{LMP}, a harmonic chain with energy-momentum conservative noise in contact with thermal baths at different temperatures is considered and it is shown that the temperature profile is given by the solution of a fractional heat equation with Dirichlet boundary conditions. In these papers the baths are of Langevin type and the fractional Laplacian which appears is not the ``restricted fractional Laplacian" like in our work but some ``spectral fractional Laplacian". We conjecture that if Langevin baths are replaced by infinite thermal baths then the macroscopic behavior is described by the ``restricted fractional Laplacian".
\end{rem}

Our second result is the following ``fractional Fick's law". 

\begin{thm}
\label{thm:Fick}
Let $\gamma \in (1,2)$ then the following fractional Fick's law holds {\footnote{The RHS of (\ref{eq:fl67}) does not depend on $x$. It can be proved by taking the derivative w.r.t. $x$ of the RHS of (\ref{eq:fl67}) and showing it vanishes thanks to (\ref{eq:fhed}).}} 
 \begin{equation}
 \label{eq:fl67}
 \begin{split}
 \lim_{N \to \infty} N^{\gamma -1} \langle W_1 \rangle_N &=c_{\gamma} \int_{-\infty}^x dy \; \int_{x}^{\infty} \, dz\, \cfrac{{\bar \rho} (y) -{\bar \rho} (z)}{(z-y)^{1+\gamma}} \;+\; \cfrac{c_\gamma}{\gamma (\gamma -1)} (\beta -\alpha)
 \end{split}
 \end{equation}
where ${\bar \rho}: \RR \to [0,1]$ is the unique solution of (\ref{eq:fhed}) and $x$ is arbitrary in $(0,1)$.
 \end{thm}

Observe that the current is a non-local function of the density.

\begin{rem} The results obtained in this paper could probably be generalized to the case where $p(\cdot)$ is such that $p(z) \sim L(z) | z|^{-(1+\gamma)}$ as $z \to \pm \infty$ for some slowly varying function $L$. Moreover, the model can be defined in higher dimensions and we expect similar results. However the proofs could be much more technical.
\end{rem}

%
%

\section{Weak solution of the stationary fractional heat equation with Dirichlet boundary conditions}
\label{sec:wsfhe}

The fractional Laplacian $(-\Delta)^{\gamma/2}$ of exponent $\gamma/2$  is defined on the set of functions $H:\RR \to \RR$ such that
\begin{equation}
\label{eq:integ1}
\int_{-\infty}^{\infty} \cfrac{|H(q)|}{(1 +|q|)^{1+\gamma}} dq < \infty
\end{equation}
by
\begin{equation}
(-\Delta)^{\gamma/2} H \, (q) = c_{\gamma}  \lim_{\ve \to 0} \int_{-\infty}^{\infty} {\bf 1}_{|y-q| \ge \ve} \, \cfrac{H(q) -H(y)}{|y-q|^{1+\gamma}} dy
\end{equation}
provided the limit exists (which is the case if $H$ is differentiable such that $H^\prime$ is $\beta$-H\"older for some $\beta>\gamma-1$ and satisfies (\ref{eq:integ1}), e.g. if $H$ is in the Schwartz space). Up to a multiplicative constant, $-(-\Delta)^{\gamma/2}$ is the generator of a $\gamma$-Levy stable process. The fractional Laplacian can also be defined in an equivalent way as a pseudo-differential operator of symbol $|\xi|^{\gamma}$ (up to a multiplicative constant). 

We are interested in the boundary problem (\ref{eq:fhed}) which has to be suitably interpreted since the fractional Laplacian is not a local operator. The correct interpretation of (\ref{eq:fhed}) which appears in this paper is that ${\bar \rho}$ is the {\textit{restriction}} to $[0,1]$ of a function $u: \RR \to \RR$ such that 
 \begin{equation}
 \begin{cases}
 \label{eq:fshe2}
 &(-\Delta)^{\gamma/2} \,  {u} (q) =0, \quad q \in (0,1),\\
 &{u} (q)=\alpha, \quad q \le 0,\\
 &{u}(q)=\beta, \quad q \ge 1.
 \end{cases}
 \end{equation}
In the PDE's literature this interpretation corresponds to the so-called ``restricted fractional Laplacian". Another popular interpretation of the fractional Laplacian with Dirichlet boundary conditions is the ``spectral fractional Laplacian" (\cite{Vaz}). The interpretation appearing in \cite{BKO} is a third one.\\

Let the functions $r^{\pm}: (0,1) \to (0, \infty)$ be defined by
\begin{equation}
r^- (q)=c_\gamma \gamma^{-1} q^{-\gamma},\quad  r^+ (q) = c_\gamma \gamma^{-1} (1-q)^{-\gamma}.
\end{equation}
The operator $\bb L$ is defined by its action on functions $H \in C^{2}_c ([0,1])$, the space of $C^2$ functions with compact suport included in $(0,1)$, by
\begin{equation}
\forall q \in (0,1), \quad ({\bb L} H)(q) = -(-\Delta)^{\gamma/2}  H \, (q) + r^{-} (q) H(q) + r^{+} (q) H(q).
\end{equation}

\begin{definition}
We say that a continuous function $\rho:[0,1] \to [0,1]$ is a weak solution of (\ref{eq:fhed}) if $\rho(0)=\alpha$, $\rho (1)=\beta$ and for any smooth function $H \in C_c^{2} ([0,1])$ we have that
\begin{equation*}
-\langle \rho \, , \,  (-\Delta)^{\gamma/2}  H \rangle + \langle \, \alpha r^- + \beta r^+   \, , \,  H \rangle =0
\end{equation*}
where $\langle \cdot, \cdot \rangle$ denotes the usual scalar product in ${\bb L}^2 ([0,1])$. 
\end{definition}

\begin{prop}
There exists a unique weak solution to (\ref{eq:fhed}). It is given by (\ref{eq:explicit}).

\end{prop}

\begin{proof}
The existence of a continuous (explicit) solution given by (\ref{eq:explicit}) and satisfying (\ref{eq:fshe2}) is a well known fact (see e.g. \cite{BB}). Let us denote it by $\rho$ and let us show it is also a weak solution.  For any $H \in C_c^2 ([0,1])$ we have
\begin{equation}
\label{eq:refcon}
\begin{split}
&\langle \rho \, , \,   {\bb L} H \rangle + \langle \alpha -\rho \, , \,  Hr^{-} \rangle + \langle \beta -\rho \, , \,  H  r^+\rangle \\  
&= -\langle \rho \, , \,   (-\Delta)^{\gamma/2} H \rangle + \langle \alpha \, , \,  Hr^{-} \rangle + \langle \beta \, , \,  H  r^+\rangle \\
&= -\int_{-\infty}^{\infty}  \rho (q) ( - \Delta)^{\gamma/2} H  (q) \, dq.
\end{split}
\end{equation}
To prove the last equality we first recall that $H$ vanishes outside of $(0,1)$, $\rho (y)=\alpha$ for $y\le 0$, $\rho (y)=\beta$ for $y\ge 1$  and 
$$r^- (q) = c_\gamma \int_{-\infty}^0 |q-y|^{-1-\gamma} dy, \quad r^+ (q)=c_\gamma \int_1^{+\infty} |q-y|^{-1-\gamma} dy.$$ 
It follows that 
\begin{equation*}
\begin{split}
&\langle \alpha \, , \,  Hr^{-} \rangle = \alpha c_\gamma \int_0^1 H(q) \, \Big(\int_{-\infty}^0 \cfrac{1}{|q-y|^{1+\gamma}} dy \Big) dq\\
&= \alpha c_\gamma \int_{-\infty}^\infty H(q) \, \Big(\int_{-\infty}^0 \cfrac{1}{|q-y|^{1+\gamma}} dy \Big) dq= c_\gamma \int_{-\infty}^0  \alpha \, \Big(\int_{-\infty}^\infty \cfrac{H(q) }{|q-y|^{1+\gamma}} dq \Big) dy\\
&=c_\gamma \int_{-\infty}^0  \, \rho (y) \Big(\int_{-\infty}^\infty \cfrac{H(q) -H(y)}{|q-y|^{1+\gamma}} dq \Big) dy =- \int_{-\infty}^0  \, \rho (y) \, (-\Delta)^{\gamma/2} H \, (y) dy
\end{split}
\end{equation*}
and similarly for $ \langle \beta \, , \,  H  r^+\rangle$. Then the last line of (\ref{eq:refcon}) follows by adding together the three terms of the second line of (\ref{eq:refcon}). Since $(-\Delta)^{\gamma/2}$ is a symmetric operator in $\LL^2 (\RR)$ we have 
\begin{equation*}
\begin{split}
& \int_{-\infty}^{\infty}  \rho (q) (  -\Delta)^{\gamma/2} H (q) \, dq = \int_{-\infty}^{\infty}  H (q) (  -\Delta)^{\gamma/2} \rho (q) \, dq\\
&= \int_{0}^{1}  H (q) (  -\Delta)^{\gamma/2} \rho (q) \, dq =0
\end{split}
\end{equation*}

Let us now turn to the uniqueness part. Let $\rho_1$ and $\rho_2$ be two weak solutions. We extend them continuously to $\RR$ by $\rho_1 (y)= \rho_2 (y)=\alpha$ if $y \le 0$ and $\rho_1 (y)= \rho_2 (y)=\beta$ if $y \ge 1$. By linearity we have that for any $H \in C^2_{c} ([0,1])$
\begin{equation*}
\langle \rho_1 -\rho_2 , (-\Delta)^{\gamma/2} H \rangle =0.
\end{equation*}
Since $\rho_1 -\rho_2 =0$ outside $(0,1)$, $\langle \cdot, \cdot \rangle$ may be replaced by the scalar product in $\LL^2 (\RR)$. By using Theorem 3.12 in \cite{BB}, there exists a $\gamma/2$-harmonic {\footnote{A function $u:(0,1) \to \RR$ is $\gamma/2$-harmonic in $(0,1)$ if for any open set $U$ with closure included in $(0,1)$ and any $x \in U$, $u(x) =\EE_x ( u (X_{\tau_U}))$ where $(X_t)_t$ is a $\gamma/2$-stable L\'evy process and $\tau_U$ is its exit time from $U$.}} (continuous) function $u$, such that $u = \rho_{1} -\rho_{2}$ a.e. Since $\rho_1 =\rho_2$ outside of $(0,1)$, we can deduce that $\rho_1=\rho_2$ everywhere.
\end{proof}

For any continuous function $F:[0,1] \to \RR$ we denote by ${\mc L}_N F$ the continuous function on $[0,1]$ obtained as the linear interpolation of the function defined by $({\mc L}_N F) (0) =({\mc L}_N F) (1)=0$ and
$$ \forall x\in \Lambda_N, \quad ({\mc L}_N F) (\tfrac{x}{N}) = \sum_{y \in \Lambda_N} p(y-x) \left[ F(\tfrac{y}{N}) -F(\tfrac{x}{N})\right].$$

We introduce also the two linear interpolation functions $r_N^{\pm} :[0,1] \to \RR$ such that for $z \in \Lambda_N$
\begin{equation}
r_N^- (\tfrac{z}{N})= \sum_{y \ge z} p(y), \quad r_N^+ (\tfrac{z}{N})= \sum_{y \le z-N} p(y)
\end{equation}   
and 
$$r_N^{\pm} (0)= r_N^{\pm} (\tfrac{1}{N}), \quad r_N^{\pm} (1)= r_N^{\pm} (\tfrac{N-1}{N}).$$

Let finally ${\mc K}_N$ the operator defined by 
\begin{equation*}
{\mc K}_N={\mc L}_N - r_N^- - r_N^+
\end{equation*}
which, for functions $F$ with compact support in $[0,1]$, satisfies
$$({\mc K}_N F)(\tfrac{x}{N}) = \sum_{y \in \ZZ} p(y-x) \left[ F(\tfrac{y}{N}) -F(\tfrac{x}{N})\right].$$

\begin{lem}
\label{lem:app}
Let $H$ be a smooth function with compact support included in $[a,1-a]$ where $a \in (0,1)$. Then we have the following uniform convergence on $[a,1-a]$ 
\begin{equation}
\begin{split}
&\lim_{N \to \infty} N^{\gamma} r_N^{-} (q) = c_\gamma \gamma^{-1} q^{-\gamma}= r^- (q),\\
&\lim_{N \to \infty} N^{\gamma} r_N^{+} (q) = c_\gamma \gamma^{-1} (1-q)^{-\gamma}=r^+ (q),\\
&\lim_{N \to \infty} N^{\gamma} ({\mc K}_N H)(q) = -\; [ (- \Delta)^{\gamma/2}  H ] \, (q).
\end{split}
\end{equation}
\end{lem}

\begin{proof}
This Lemma establishes uniform convergence of Riemann sums to corresponding integrals. But since the uniformity statement requires a bit of technical work it is postponed to Appendix \ref{sec:app3}.
 \end{proof}

\begin{rem}
\label{rem:refcon} The two first items of the previous lemma are in fact valid for $\gamma \in (0,2)$. See the proof in Appendix \ref{sec:app3}.
\end{rem}

\section{Proofs}
\label{sec:Proofs}

The first step consists to obtain a sharp upper bound on the average current in the non-equilibrium stationary state (see Lemma \ref{lem:Fick}). This bound will be used to derive an estimate of the entropy production (Lemma \ref{lem:entropyproductionbound}) which is the key estimate to obtain by a coarse graining argument and entropy bounds that the empirical density at each extremity of $\Lambda_N$ is given by $\alpha$ and $\beta$ (Corollary \ref{cor:boundary}). To identify the form of the stationary profile in the bulk, we use a method introduced in \cite{KLO} for boundary driven diffusive systems (Lemma \ref{lem:statprof}). Fractional Fick's law is then derived.

\subsection{Entropy production bounds}

\begin{lem}
\label{lem:Fick}
Let $\gamma \in (1,2)$. There exists a constant $C>0$ such that for any $N\ge 2$ 
$$\langle W_1 \rangle_{N} \le C N^{1-\gamma}.$$ 
\end{lem}

\begin{proof}
By stationarity we have that for any $x \in \Lambda_N$, $\langle W_1 \rangle_{N} = \langle W_x \rangle_N$. It follows that
\begin{equation}
\begin{split}
\langle W_1 \rangle_{N} = \cfrac{1}{N-1} \sum_{x=1}^{N-1} \langle W_x \rangle_N & = \cfrac{1}{N-1} \sum_{y<z} p(z-y) [\langle \eta_y \rangle_N -\langle \eta_z \rangle_N ] \theta (y,z)\\
&+ (\beta- \alpha)  \sum_{\substack{y\le 0\\ z \ge N}} p(z-y)
\end{split}
\end{equation}
where 
$$\theta (y,z) ={\rm{Card}} \{ x \in \Lambda_N \; ;\; y+1 \le x \le z \}.$$
Considering the different positions of $y,z$ in $\Lambda_N$, we get 
\begin{equation}
\label{eq:WWW}
\begin{split}
\langle W_1 \rangle_{N} &= \cfrac{1}{N-1} \sum_{z=1}^{N-1} z [ \alpha- \langle \eta_z \rangle_N ] \sum_{y \le 0} p(z-y) \\
&+\cfrac{1}{N-1} \sum_{y=1}^{N-1} (N-1 -y) [ \langle \eta_y \rangle_N -\beta ] \sum_{z \ge N} p(z-y) \\
&+\cfrac{1}{N-1} \sum_{\substack{y<z \\ z,y \in \Lambda_N}} p(z-y) (z-y) [\langle \eta_y \rangle_N -\langle \eta_z \rangle_N]\\
&=(I)+(II)+(III).  
\end{split}
\end{equation}
We have that 
\begin{equation*}
|(I)| \le \cfrac{2}{N-1} \sum_{z=1}^{N-1} z \sum_{y \ge z} p(y) = {\mc O} (N^{1-\gamma}) 
\end{equation*}
since $\sum_{y \ge z} p(y) = {\mc O} (z^{-\gamma})$ as $z \to \infty$. A similar upper bound is valid for $(II)$. For the last term we observe that
\begin{equation*}
\begin{split}
(III)&=-\cfrac{1}{N-1} \sum_{y=1}^{N-2}  \sum_{k=1}^{N-1-y} k p(k)[ \langle \eta_{y+k} \rangle_N - \langle \eta_y \rangle_N ].\\
\end{split}
\end{equation*}
Now, using Fubini Theorem we get
\begin{equation*}
(III)=-\cfrac{1}{N-1} \sum_{k=1}^{N-2} kp(k) \sum_{y=1}^{N-1-k} [ \langle \eta_{y+k} \rangle_N - \langle \eta_y \rangle_N ].
\end{equation*}
Observe that for any sequence $(f(x))_x$ and any $n, k\ge 1$ we have 
$$\sum _{x =1}^{n} [f(x+k)-f(x)] = \sum_{x=1}^{k} [f(n +1+k -x)-f(x)].$$
It follows that 
\begin{equation*}
\begin{split}
(III)&=-\cfrac{1}{N-1} \sum_{k=1}^{N-2} kp(k) \sum_{y=1}^{k} [ \langle \eta_{N-y} \rangle_N - \langle \eta_y \rangle_N ]
\end{split}
\end{equation*}
so that 
$$ |(III)| \le \cfrac{2}{N-1} \sum _{k=1}^{N-2} k^{2}p(k)= {\mc O} (N^{1-\gamma}).$$ 
\end{proof}

A simple consequence of this Lemma is the following bound on the Dirichlet forms of the stationary state. 

\begin{lem}
\label{lem:entropyproductionbound}
Let $\rho \in (0,1)$. There exists a constant $C:=C(\rho, \alpha,\beta)>0$ such that for any $N \ge 2$
\begin{equation*}
\begin{split}
& \sum_{x, y \in \Lambda_N} p(y-x) \left\langle \left[ \sqrt{f_{N,\rho} (\eta^{xy})} - \sqrt{f_{N,\rho} (\eta)} \right]^2\right\rangle_{\rho} \; \le \; \cfrac{C}{N^{\gamma -1}},\\
& \sum_{x \in \Lambda_N} \sum_{y \le 0} p(y-x) \left\langle \left[ {\sqrt {f_{N,\alpha} (\eta^{x})}} -{\sqrt {f_{N,\alpha} (\eta)}} \right]^2\right\rangle_\alpha \; \le \; \cfrac{C}{N^{\gamma -1}},\\
& \sum_{x \in \Lambda_N} \sum_{y \ge N} p(y-x) \left\langle \left[ {\sqrt {f_{N,\beta} (\eta^{x})}} -{\sqrt {f_{N,\beta} (\eta)}} \right]^2\right\rangle_\beta \; \le \; \cfrac{C}{N^{\gamma -1}}.\\
 \end{split}
 \end{equation*}
\end{lem}

\begin{proof}
To simplify the notation we denote $f_{N,\rho}$ by $f_N$. By definition of stationary state we have:
\begin{equation}
\label{eq:byby0}
\begin{split}
& 0 = \langle f_{N} L_{N} \log f_{N}\rangle_{\rho}\\
&=  \langle f_{N} L^0_N \log f_{N}\rangle_{\rho}+ \langle f_{N} L_N^{r} \log f_{N}\rangle_{\rho}+\langle f_{N} L_N^{\ell}  \log f_{N}\rangle_{\rho}.
\end{split}
\end{equation}
We first obtain an upper bound for the second and the third term on the right hand side of the previous equality. For any $R>0$, the second term is equal to
\begin{equation}
\label{eq:second}
\begin{split}
 &\sum_{\substack{x \in \Lambda_N\\y \ge N}}p(x-y)\langle f_{N}(\eta)\eta_{x} (1-\beta) \left[ \log f_{N}(\eta^{x})-\log f_{N}(\eta)\right]\rangle_{\rho}\\
& +\sum_{\substack{x \in \Lambda_N\\y \ge N}}p(x-y)\langle f_{N}(\eta)(1-\eta_{x})\beta \left[ \log f_{N}(\eta^{x})-\log f_{N}(\eta)\right]\rangle_{\rho}\\
&= \sum_{\substack{x \in \Lambda_N\\y \ge N}}p(x-y)\left\langle f_{N}(\eta)\eta_{x} (1-\beta) \left[ \log\dfrac{R f_{N}(\eta^{x})}{f_{N}(\eta)} \right]\right\rangle_{\rho}\\
& +\sum_{\substack{x \in \Lambda_N\\y \ge N}}p(x-y)\left\langle f_{N}(\eta)(1-\eta_{x})\beta \left[ \log\dfrac{f_{N}(\eta^{x})}{R f_{N}(\eta)} \right]\right\rangle_{\rho}\\
&- \log R \sum_{\substack{x \in \Lambda_N\\y \ge N}}p(x-y)\left\langle f_{N}(\eta)\left(\eta_{x} (1-\beta)-(1-\eta_{x})\beta\right)\right\rangle_{\rho}.
\end{split}
\end{equation}
Now by the change of variable $w = \eta^{x}$ we have that (\ref{eq:second}) is equal to 
\begin{equation*}
\begin{split}
 &-\sum_{\substack{x \in \Lambda_N\\y \ge N}}p(x-y)\left\langle f_{N}(w^{x})(1-w_{x}) (1-\beta) \left[\log\dfrac{f_{N}(w^{x})}{Rf_{N}(w)} \right]\left(\dfrac{\rho}{1-\rho}\right)\right\rangle_{\rho}\\
 & +\sum_{\substack{x \in \Lambda_N\\y \ge N}}p(x-y)\left\langle f_{N}(\eta)(1-\eta_{x})\beta \left[ \log\dfrac{f_{N}(\eta^{x})}{R f_{N}(\eta)} \right]\right\rangle_{\rho}\\
&- \log R \sum_{\substack{x \in \Lambda_N\\y \ge N}}p(x-y)\left\langle f_{N}(\eta)\left(\eta_{x} (1-\beta)-(1-\eta_{x})\beta\right)\right\rangle_{\rho}
\end{split}
\end{equation*}
Now, choosing $R = \dfrac{\beta}{1-\beta}\dfrac{1-\rho}{\rho}$ and using $(x-y)\log(y/x)< 0$, we have that the last expression is equal to 
\begin{eqnarray*}
 &&\dfrac{\beta}{R}\sum_{\substack{x \in \Lambda_N\\y \ge N}}p(x-y)\left\langle (1-w_{x})\left( Rf_{N}(w)-f_{N}(w^{x}) \right)\left[\log\dfrac{f_{N}(w^{x})}{Rf_{N}(w)} \right]\right\rangle_{\rho}\\
&&- \log R \sum_{\substack{x \in \Lambda_N\\y \ge N}}p(x-y)\left\langle f_{N}(\eta)\left(\eta_{x} (1-\beta)-(1-\eta_{x})\beta\right)\right\rangle_{\rho}\\
&\le & - \log \left(\dfrac{\beta}{1-\beta}\dfrac{1-\rho}{\rho} \right)\sum_{\substack{x \in \Lambda_N\\y \ge N}}p(x-y)\left\langle f_{N}(\eta)\left(\eta_{x}-\beta\right)\right\rangle_{\rho}.
\end{eqnarray*}
We proved therefore that 
$$ \langle f_{N}L_N^{r}\log f_{N}\rangle_{\rho} \le -\log \left(\dfrac{\beta}{1-\beta}\dfrac{1-\rho}{\rho} \right) \left\langle W_{N}\right\rangle_{N}.$$
Similar computations give that
$$  \langle f_{N}L_N^{\ell} \log f_{N}\rangle_{\rho} \le -\log \left(\dfrac{1-\alpha}{\alpha}\dfrac{\rho}{1-\rho} \right)\left\langle W_{1}\right\rangle_{N}.$$
By Lemma \ref{lem:Fick}, we get that there exists a constant $C^{\prime}> 0$ such that
$$ \langle f_{N}L_N^{r}\log f_{N}\rangle_{\rho} \le  C^{\prime} N^{1-\gamma}, $$
$$  \langle f_{N}L_N^{\ell} \log f_{N}\rangle_{\rho} \le C^{\prime} N^{1-\gamma}.$$
Therefore, by (\ref{eq:byby0}), we have that
$$ -\langle f_{N}L_N^{0} \log f_{N}\rangle_{\rho} \le C N^{1-\gamma}.$$
Now, using the simple inequality $a(\log b -\log a)\leq 2\sqrt{a}(\sqrt{b}-\sqrt{a})$, we obtain that
$$ -\langle \sqrt{f_{N}}L_N^{0} \sqrt{f_{N}}\rangle_{\rho} \le C N^{1-\gamma}.$$
This gives the first inequality in Lemma \ref{lem:entropyproductionbound} since the left hand side of the previous inequality is equal to the left hand side of the first inequality of Lemma \ref{lem:entropyproductionbound} because $L_N^0$ is reversible with respect to $\nu_{\rho}$ for any $\rho$. Choosing now $\rho=\alpha$, and using again the simple inequality $a(\log b -\log a)\leq 2\sqrt{a}(\sqrt{b}-\sqrt{a})$, we have that
$$ -\langle \sqrt{f_{N, \alpha}}\; L_N^{\ell} \sqrt{f_{N,\alpha}}\rangle_{\alpha} \le  C^{\prime} N^{1-\gamma}.$$
Since $L_N^\ell$ is reversible with respect to $\nu_{\alpha}$ we have that 
\begin{equation*}
\begin{split}
&-\langle \sqrt{f_{N, \alpha}}\; L_N^{\ell} \sqrt{f_{N,\alpha}}\rangle_{\alpha} \\
&= \cfrac{1}{2} \sum_{x \in \Lambda_N} \sum_{y \le 0} p(y-x) \left\langle \left[ \eta_x (1-\alpha) + (1-\eta_x) \alpha \right] \left[ {\sqrt {f_{N,\alpha} (\eta^{x})}} -{\sqrt {f_{N,\alpha} (\eta)}} \right]^2\right\rangle_\alpha.
\end{split}
\end{equation*}
Since $\alpha \wedge \, 1-\alpha \le \eta_x (1-\alpha) + (1-\eta_x) \alpha $, the term above is bigger or equal to a constant times the left hand side of the second inequality of Lemma \ref{lem:entropyproductionbound}. The third inequality of Lemma \ref{lem:entropyproductionbound} is obtained similarly by choosing $\rho=\beta$. 
\end{proof}

\subsection{Proof of Theorem \ref{thm:weak-profile} }

Let ${\mc M}^+_{d}$, $d=1,2$,  be the space of positive measures on $[0,1]^d$ with total mass bounded by $1$ equipped with the weak topology. For any $\eta \in \Omega^N$ the empirical measures $\pi^N (\eta) \in {\mc M}_1^+$ (resp. ${\hat \pi}^N (\eta) \in {\mc M}_2^+$) is defined by 
\begin{equation}
{\pi}^N (\eta) =\cfrac{1}{N-1} \sum_{x=1}^{N-1} \eta_x \delta_{x/N}
\end{equation}
resp.
\begin{equation}
{\hat \pi}^N (\eta) =\cfrac{1}{(N-1)^2} \sum_{x,y=1}^{N-1} \eta_x \eta_y  \delta_{(x/N,y/N)}
\end{equation}
where $\delta_u$ (resp. $\delta_{(u,v)}$) is the Dirac mass on $u \in [0,1]$ (resp. $(u,v) \in [0,1]^2$). Let ${\PP}^N$ be the law on ${\mc M}_1^+ \times {\mc M}_2^+$ induced by $(\pi^N, {\hat \pi}^N) : \Omega^N \to {\mc M}_1^+ \times {\mc M}_2^+$ when $\Omega^N$ is equipped with the non-equilibrium stationary state $\mu_N$. To simplify notations, we denote $\pi^N (\eta)$ (resp. ${\hat \pi}^N (\eta)$) by $\pi^N$ (resp. ${\hat \pi}^N$) and the action of $\pi \in {\mc M}_d^+$ on a continuous function $H:[0,1]^d \to \RR$ by $\langle \pi , H \rangle = \int_{[0,1]^d} H(u) \pi (du)$.  

The sequence $(\PP^N)_{N\ge 2}$ is tight on ${\mc M}_1^+ \times {\mc M}_2^+$. This is obvious since it is a family of probabilities over the compact set ${\mc M}_1^+ \times {\mc M}_2^+$. Our goal is to prove that every limit point $\PP^*$ of this sequence is concentrated on the set of measures $(\pi ,{\hat \pi})$ of ${\mc M}_1^+ \times {\mc M}_2^+$ such that $\pi$ (resp. ${\hat \pi}$) is absolutely continuous with respect to the Lebesgue measure on $[0,1]$ (resp. $[0,1]^2$) and with a density  $\rho(u)$ (resp. $\rho(u) \rho(v)$) where $\rho$ is a weak solution of (\ref{eq:fhed}). 

\begin{lem}
\label{lem:ghj}
Let $\PP^*$ be a limit point of the sequence $(\PP^N)_N$. Then $\PP^*$ is concentrated on measures $(\pi,{\hat \pi} )$ such that $\pi$ (resp. $\hat \pi$) is absolutely continuous with respect to Lebesgue measure on $[0,1]$ (resp. $[0,1]^2$). The density $\rho$ of $\pi$ is a continuous function on $[0,1]$ and the density of $\hat \pi$ is equal to $\rho \otimes \rho :(x,y) \in [0,1]^2 \to \rho(x) \rho (y)$.
\end{lem}

\begin{proof} See Appendix \ref{app:B}.
\end{proof}

With some abuse of notation we denote  by $({\PP^N})_N$ a fixed subsequence converging to a limit point $\PP^*$. A generic element of ${\mc M}_1^+ \times {\mc M}_2^+$ is denoted by $(\pi, {\hat \pi})$ with the convention that $\pi$ and $\hat \pi=\pi \otimes \pi$ denotes the probability measure as well as its density with respect to the Lebesgue measure.


\begin{prop}
\label{cor:boundary}
We have that $\PP^{*}$ almost surely ${\pi} (0) =\alpha$ and $\pi (1) =\beta$.
\end{prop}

\begin{proof}

For small $\ve>0$ and small $\lambda \in \RR$, let $B$ be the box $B:=\{ [N\ve], \ldots, N-1\}$ in $\Lambda_{N}$ and let $u$ be the function defined by 
\begin{equation*}
u=e^{\lambda \sum_{x \in B} \eta_x}.
\end{equation*} 
We recall that the action of the generator $L_N^{\ell}$ on a function $f :\Omega_N \to \RR$ can be rewritten as 
\begin{equation*}
(L_N^\ell f)(\eta) = \sum_{z \in \Lambda_N} r_N^{-} \big( \tfrac{z}{N} \big) \big[ \eta_z (1-\alpha) + (1-\eta_z) \alpha \big] \big[ f(\eta^z) -f(\eta) \big]
\end{equation*}
where $r_N^{-} \big( \tfrac{z}{N} \big) = \sum_{y \ge z} p(y)$. An elementary computation shows that 
\begin{equation}
\label{eq:mil}
\begin{split}
- \cfrac{L_N^\ell u}{u} &= \left[ (e^\lambda -1) -2(1-\alpha)(\cosh \lambda -1) \right] \sum_{z \in B}  r_N^{-} \big( \tfrac{z}{N} \big) (\eta_z -\alpha)\\
&-2 \alpha (1-\alpha) (\cosh \lambda -1) \sum_{z \in B}  r_N^{-} \big( \tfrac{z}{N} \big).
\end{split}
\end{equation}
Multiplying (\ref{eq:mil}) by $f_{N, \alpha}$, integrating w.r.t. $\nu_{\alpha}$ and using the variational formula of the Dirichlet form we deduce that
\begin{equation}
\begin{split}
&\left[ (e^\lambda -1) -2(1-\alpha)(\cosh \lambda -1) \right] \sum_{z \in B}  r_N^{-} \big( \tfrac{z}{N} \big) (\langle \eta_z \rangle_N  -\alpha)\\
&\le \sum_{z \in \Lambda_N}  r_N^{-} \big( \tfrac{z}{N} \big) \left\langle \left[ \sqrt{f_{N,\alpha} (\eta^z)} - \sqrt{f_{N,\alpha} (\eta)} \right]^2 \right\rangle_{\alpha}\\
&+2 \alpha (1-\alpha) (\cosh \lambda -1) \sum_{z \in B}  r_N^{-} \big( \tfrac{z}{N} \big)\\
& \le CN^{1-\gamma} + 2 \alpha (1-\alpha) (\cosh \lambda -1) \sum_{z \in B}  r_N^{-} \big( \tfrac{z}{N} \big)
\end{split}
\end{equation}
where the last inequality is a consequence of Lemma \ref{lem:entropyproductionbound}. Observe that for $\lambda \to 0$, the term $(e^\lambda -1) -2(1-\alpha)(\cosh \lambda -1)$ is equivalent to $\lambda$ and has therefore the sign of $\lambda$ for sufficiently small $\lambda$. The term $\cosh \lambda -1$ is of order $\lambda^2$. 
Assume first that $\lambda>0$ is small. Then there exists a constant $C>0$ independent of $\lambda, \ve$ and $N$ such that
\begin{equation}
\begin{split}
&\mu_N \left(  \langle \pi_N - \alpha \, , \, N^{\gamma} \, {\bf 1}_{[\ve ,1]} \big(\tfrac{z}{N} \big) \,r_N^{-} \big( \tfrac{z}{N} \big) \rangle \right) = N^{\gamma-1} \sum_{z \in B}  r_N^{-} \big( \tfrac{z}{N} \big) (\langle \eta_z \rangle_N  -\alpha)\\
&\le \cfrac{C}{\lambda} + C \lambda N^{-1} \sum_{z \in B} N^{\gamma} r_N^{-} \big( \tfrac{z}{N} \big).
\end{split}
\end{equation}
By Lemma \ref{lem:app} we have that for some constant $C>0$
\begin{equation}
N^{-1} \sum_{z \in B} N^{\gamma} r_N^{-} \big( \tfrac{z}{N} \big) \le C \int_{\ve}^{1} q^{-\gamma} dq = {\mc O} (\ve^{1-\gamma}).
\end{equation}
Therefore we conclude that
\begin{equation}
\limsup_{\ve \to 0} \ve^{\gamma -1} \limsup_{N \to \infty} \mu_N \left(  \langle \pi_N - \alpha \, , \, {\bf 1}_{[\ve ,1]} \big(\tfrac{z}{N} \big) N^{\gamma} r_N^{-} \big( \tfrac{z}{N} \big) \rangle \right) \le 0.
\end{equation}
Similarly, by considering small $\lambda <0$, we deduce that
\begin{equation}
\liminf_{\ve \to 0} \ve^{\gamma -1} \liminf_{N \to \infty} \mu_N \left(  \langle \pi_N - \alpha \, , \, {\bf 1}_{[\ve ,1]} \big(\tfrac{z}{N} \big) N^{\gamma} r_N^{-1} \big( \tfrac{z}{N} \big) \rangle \right) \ge 0.
\end{equation}
By using Lemma \ref{lem:app} we deduce that $\PP^*$ a.s. we have 
\begin{equation}
\lim_{\ve \to 0} \ve^{\gamma -1} \int_{\ve}^1 \cfrac{\pi (q) - \alpha}{q^\gamma} \, dq =0.
\end{equation}
But since by Lemma \ref{lem:ghj} $\pi$ is a continuous function on $[0,1]$, if $\pi(0) \ne \alpha$, we have that 
$$\lim_{\ve \to 0} \ve^{\gamma -1} \int_{\ve}^1 \cfrac{\pi (q) - \alpha}{q^\gamma} \, dq=\cfrac{\pi(0) -\alpha}{\gamma -1} \ne 0$$
and we get a contradiction. We deduce thus that $\pi(0)=\alpha$. Similarly $\pi (1) =\beta$.
\end{proof}

\begin{rem}
The usual proof for driven diffusive systems of this proposition is quite different and based on the so-called two-blocks estimate (\cite{ELS}, \cite{KLO2}). It turns out that in the context of exclusion process with long jumps this approach does not work since the control of the entropy production is not sufficient to cancel the heavy tails of $p$, even by using the arguments of \cite{Jara1}. 
 \end{rem}

\begin{lem}
\label{lem:statprof}
Let $\bar \rho$ be the unique weak solution of (\ref{eq:fhed}). For any $F,G$ in $C_c^{\infty} ([0,1])$ we have
\begin{equation}
\int_{[0,1]^2} \left[ G(u) ((-\Delta)^{\gamma/2} F) (v) \, + \, F(v) ((-\Delta)^{\gamma/2}G) (u) \right] I(u,v) du dv =0
\end{equation}
where
\begin{equation}
I(u,v) = {\bb E}^* \left[ (\pi (u) -{\bar \rho} (u)) (\pi (v) -{\bar \rho} (v)) \right].
\end{equation}
\end{lem}

\begin{proof}
We have that
\begin{equation}
\label{eq:12345}
\begin{split}
&L_N ( \langle \pi^N, F \rangle ) = \cfrac{1}{N-1} \sum_{x\in \Lambda_N} \sum_{y \in \ZZ} F(\tfrac{x}{N}) p(y-x) (\eta_y -\eta_x) \\
&=  \langle \pi^N, {\mc K}_N F \rangle + \cfrac{\alpha}{N-1} \sum_{x \in \Lambda_N}  (F r_N^-)(\tfrac{x}{N}) + \cfrac{\beta}{N-1}  \sum_{x \in \Lambda_N}  (F r_N^+)(\tfrac{x}{N})
\end{split}
\end{equation}
where 
$${\mc K}_N F = {\mc L}_N F - F r_N^- - F r_N^+.$$

We then multiply (\ref{eq:12345}) by $N^\gamma$ and take the expectation with respect to $\mu_N$ on both sides, the left hand side being then equal to $0$ by stationarity. By using Lemma \ref{lem:app} and weak convergence we conclude that
\begin{equation*}
{\bb E}^* \left[ \int_0^1 \left\{  {\bb L} F - r^{-} F - r^+ F \right\}  (x) \;  \pi (x) dx  \right] +\int_0^1 \left\{  \alpha r^{-} F + \beta  r^+ F \right\}  (x) \;  dx = 0.
\end{equation*}

We compute now $L_N (\langle {\hat \pi}^N, J\rangle)$ where $J:[0,1]^2 \to \RR$ is a smooth test function with compact support strictly included in $[0,1]^2$ and which is identically equal to $0$ on the diagonal. Consider a small $\delta>0$ and take a smooth even function $H_\delta : \RR \to [0,1]$ which is equal to $0$ on $[-\delta, \delta]$ and equal to $1$ outside of $[-2\delta, 2 \delta]$. Let then $J_\delta (u,v) = F(u) G(v) H_{\delta} (v-u)$, $(u,v) \in [0,1]^2$.   

For $u \in [0,1]$ we denote by $F_{\delta,u}, G_{\delta, u}$ the functions given by
\begin{equation*}
\begin{split}
& F_{\delta,u} (v) = F(v)H_{\delta} (v-u), \quad  G_{\delta,u} (v) = G(v)H_{\delta} (v-u).
\end{split}
\end{equation*}

By using Lemma \ref{lem:compA} we get that
\begin{equation}
\label{eq:sou}
\begin{split}
L_N (\langle {\hat \pi}^N, J_\delta\rangle)&= \cfrac{1}{N-1} \sum_{x \in \Lambda_N} \eta_x F(\tfrac{x}{N})  \langle \pi^N, {\mc K}_N G_{\delta, x/N} \rangle\\
& + \cfrac{1}{N-1} \sum_{x \in \Lambda_N} \eta_x G(\tfrac{x}{N})  \langle \pi^N, {\mc K}_N F_{\delta, x/N} \rangle\\
&+ \cfrac{\alpha}{N-1} \sum_{x \in \Lambda_N} \eta_x G(\tfrac{x}{N})  \left\{ \cfrac{1}{N-1} \sum_{y \in \Lambda_N} F_{\delta, x/N} (\tfrac{y}{N}) r_N^{-} (\tfrac{y}{N}) \right\}\\
&+ \cfrac{\alpha}{N-1} \sum_{x \in \Lambda_N} \eta_x F(\tfrac{x}{N})  \left\{ \cfrac{1}{N-1} \sum_{y \in \Lambda_N} G_{\delta, x/N} (\tfrac{y}{N}) r_N^{-} (\tfrac{y}{N}) \right\}\\
&+ \cfrac{\beta}{N-1} \sum_{x \in \Lambda_N} \eta_x G(\tfrac{x}{N})  \left\{ \cfrac{1}{N-1} \sum_{y \in \Lambda_N} F_{\delta, x/N} (\tfrac{y}{N}) r_N^{+} (\tfrac{y}{N}) \right\}\\
&+ \cfrac{\beta}{N-1} \sum_{x \in \Lambda_N} \eta_x F(\tfrac{x}{N})  \left\{ \cfrac{1}{N-1} \sum_{y \in \Lambda_N} G_{\delta, x/N} (\tfrac{y}{N}) r_N^{+} (\tfrac{y}{N}) \right\}\\
&-\cfrac{1}{(N-1)^2} \sum_{x,y \in \Lambda_N} p(y-x) (\eta_y -\eta_x)^2 J_{\delta} (\tfrac xN \tfrac yN).
 \end{split}
\end{equation}

Since $J_{\delta} (u,v)$ is equal to $0$ for $|u -v| \le \delta$, we have that
\begin{equation*}
 N^{\gamma} \mu_N \left( \cfrac{-1}{(N-1)^2} \sum_{x,y \in \Lambda_N} p(y-x) (\eta_y -\eta_x)^2 J_{\delta} (\tfrac xN \tfrac yN)\right) = {\mc O} (N^{-1}).
\end{equation*} 

We multiply (\ref{eq:sou}) by $N^\gamma$ and take the expectation with respect to $\mu_N$ on both sides, the left hand side being then equal to $0$ by stationarity. By using Lemma \ref{lem:app} and weak convergence we conclude that
\begin{equation*}
\begin{split}
&-{\bb E}^* \left[ \int_{[0,1]^2} \left\{ G(u) ( (-\Delta)^{\gamma/2} F_{\delta,u}) (v) \, + \, F(v) ((-\Delta)^{\gamma/2} G_{\delta,v}) (u) \right\}  \pi(u) \pi (v) du dv \right] \\
&+ {\bb E}^* \left[ \int_{[0,1]^2} \left\{ G(u) \, \alpha r^{-} (v) F_{\delta,u} (v) \, + \, G(u) \, \beta r^+(v) F_{\delta,u} (v) \right\}  \pi(u) du dv \right]\\
&+ {\bb E}^* \left[ \int_{[0,1]^2} \left\{ F(u) \, \alpha r^{-} (v) G_{\delta,u} (v) \, + \, F(u) \, \beta r^+(v) G_{\delta,u} (v) \right\}  \pi(u) du dv \right]=0.
\end{split} 
\end{equation*}

We can take the limit $\delta \to 0$ and since $H_{\delta}$ converges to the function identically equal to $1$, we get 
\begin{equation*}
\begin{split}
&-{\bb E}^* \left[ \int_{[0,1]^2} \left\{ G(u) ( (-\Delta)^{\gamma/2} F) (v) \, + \, F(v) ((-\Delta)^{\gamma/2} G) (u) \right\}  \pi(u) \pi (v) du dv \right] \\
&+ {\bb E}^* \left[ \int_{[0,1]^2} \left\{ G(u) \, \alpha r^{-} (v) F (v) \, + \, G(u) \, \beta r^+(v) F (v) \right\}  \pi(u) du dv \right]\\
&+ {\bb E}^* \left[ \int_{[0,1]^2} \left\{ F(u) \, \alpha r^{-} (v) G (v) \, + \, F(u) \, \beta r^+(v) G (v) \right\}  \pi(u) du dv \right]=0.
\end{split} 
\end{equation*}

We have also proved that for any smooth compactly supported function $H$
\begin{equation*}
-{\bb E}^* \left[ \int_{0}^1 ((-\Delta)^{\gamma/2}  H) (u)  \pi(u)  du \right] +\int_0^1 \left\{  \alpha r^{-} H + \beta  r^+ H \right\}  (u) \;  du=0.
\end{equation*}
Let $\bar \rho$ be the unique weak solution of (\ref{eq:fhed}). Then we have
\begin{equation*}
-\int_{0}^1 ((-\Delta)^{\gamma/2}  H) (u)  {\bar \rho}(u)  du +\int_0^1 \left\{  \alpha r^{-} H + \beta  r^+ H \right\}  (u) \;  du=0.
\end{equation*}
It follows that
\begin{equation}
\int_{[0,1]^2} \left[ G(u) ((-\Delta)^{\gamma/2} F) (v) \, + \, F(v) ((-\Delta)^{\gamma/2}G) (u) \right] I(u,v) du dv =0.
\end{equation}
\end{proof}

Since $\PP^*$ almost surely $\pi (0) = \bar \rho (0) =\alpha$ and $\pi (1) = \bar\rho (1) =\beta$ and that $\pi, \bar \rho$ are continuous functions, by extending then to $\RR$ by $\pi (x) ={\bar \rho} (x) =\alpha$ if $x \le 0$ and ${\pi} (x) ={\bar \rho} (x) =\beta$ if $x \geq 1$, we get that for any $F,G$ in $C_c^{\infty} ([0,1]^2)$, 
\begin{equation}
\int_{\RR^2} \left[ G(u) ((-\Delta)^{\gamma/2} F) (v) \, + \, F(v) ((-\Delta)^{\gamma/2}G) (u) \right] I(u,v) du dv =0
\end{equation}
By using Theorem 3.12 in \cite{BB} we deduce that $I$ is a.s. constant with respect to Lebesgue measure on $[0,1]^2$. Since by Proposition \ref{cor:boundary}, we have $I(0,0)=I(1,1)=0$, we deduce that $I$ is identically equal to $0$. Thus $\PP^*$ almost surely $\pi ={\bar \rho}$.

We have proved
\begin{prop}
The sequence $(\PP^N)_N$ converges in law to the delta measure concentrated on 
$$({\bar \rho} (x) dx, {\bar \rho} (x) {\bar \rho} (y) dx dy)$$
where $\bar\rho$ is the unique weak solution of (\ref{eq:fhed}).
\end{prop}

Theorem \ref{thm:weak-profile} is a trivial consequence of this proposition. 

\subsection{Proof of Fick's law}

Let us define for $z \in \Lambda_N$ 
$$ {\tilde  r}_N^- \big(\tfrac{z}{N}\big) = \sum_{y\ge z} y p(y), \quad  {\tilde  r}_N^+ \big(\tfrac{z}{N}\big) = - \sum_{y \le z-N} y p(y)$$
which are, up to  a multiplicative constant, defined as $r_N^{\pm}$ with $\gamma$ replaced by $\gamma-1 \in (0,1)$.
Recalling (\ref{eq:WWW}) we see that
\begin{equation*}
N^{\gamma -1} \langle W_1 \rangle_N= \mu_N \left( \, \langle  \pi^N, \varphi_N \rangle\, \right)  +N^{\gamma -1} \theta_N
\end{equation*}
where $\varphi_N:(0,1) \to \RR$ defined by
\begin{equation*}
\begin{split}
\varphi_N(\tfrac{z}{N}) &= - N^{\gamma} \sum_{y \le 0} \tfrac{z}{N} \, p(z-y) +N^{\gamma}  \sum_{y \ge N} \big[ 1-\tfrac{1}{N} -\tfrac{z}{N} \big]\, p(y-z)\\
&+N^\gamma \sum_{\substack{y>z\\y\in \Lambda_N}} p(y-z) \big(\tfrac{y-z}{N}\big) - N^\gamma \sum_{\substack{y<z\\y\in \Lambda_N}} p(y-z)  \big(\tfrac{z-y}{N}\big)\\
&=- \tfrac{z}{N} N^\gamma r_N^- \big(\tfrac{z}{N}\big) + \big( 1-\tfrac{1}{N}-\tfrac{z}{N}\big) N^\gamma r_N^+ \big(\tfrac{z}{N} \big) + N^\gamma \sum_{y\in \Lambda_N} p(y-z) \big(\tfrac{y-z}{N}\big)\\
&=- \tfrac{z}{N} N^\gamma r_N^- \big(\tfrac{z}{N}\big) + \big( 1-\tfrac{1}{N}-\tfrac{z}{N}\big) N^\gamma r_N^+ \big(\tfrac{z}{N} \big) +N^{\gamma-1} {\tilde  r}_N^- \big(\tfrac{z}{N}\big)-N^{\gamma-1} {\tilde  r}_N^+ \big(\tfrac{z}{N}\big)
\end{split}
\end{equation*}
is a discrete approximation of the function $\varphi: (0,1) \to \RR$ given by
\begin{equation*}
\begin{split}
\varphi (q) &=\,  \cfrac{c_{\gamma}}{\gamma (1-\gamma)} \, \{ (1-q)^{1-\gamma} - q^{1-\gamma} \}\\
\end{split}
\end{equation*}
and 
\begin{equation*}
\begin{split}
\theta_N &= \cfrac{\alpha}{N-1} \sum_{z=1}^{N-1} \sum_{y \le 0} z \, p(z-y) - \cfrac{\beta}{N-1} \sum_{y=1}^{N-1} \sum_{z \ge N} (N-1-y)\,  p(z-y).
\end{split}
\end{equation*}

It is easy to compute the limit of $N^{\gamma -1}\theta_N$ by writing it as a Riemann sum:
\begin{equation*}
\begin{split}
\lim_{N \to \infty} N^{\gamma -1} \theta_N &=  \alpha c_\gamma \; \lim_{N \to \infty}  \cfrac{N}{N-1} \cfrac{1}{N^2} \sum_{z=1}^{N-1} \sum_{y\le 0} \cfrac{\tfrac{z}{N}}{\big| \tfrac{z}{N} -\tfrac{y}{N} \big|^{1+\gamma}} \\
&- \beta c_\gamma  \; \lim_{N \to \infty}  \cfrac{N}{N-1} \cfrac{1}{N^2} \sum_{y=1}^{N-1} \sum_{z\ge N} \cfrac{(1-\tfrac{1}{N} -\tfrac{y}{N})}{\big| \tfrac{z}{N} -\tfrac{y}{N} \big|^{1+\gamma}} \\
&= \alpha c_\gamma \; \int_0^1 \Big(\int_{-\infty}^0\tfrac{dy}{|z-y|^{1+\gamma}}\Big) \, z \, dz -\beta c_\gamma \; \int_0^1 \Big(\int_{1}^{+\infty} \tfrac{dz}{|z-y|^{1+\gamma}}\Big) \, (1-y) \, dy\\
&= \, \cfrac{c_\gamma (\alpha -\beta)}{\gamma (2-\gamma)}.
\end{split}
\end{equation*}

Let us now compute the limit of $\mu_N \Big( \langle \pi^N, \varphi_N \rangle \Big)= \tfrac{1}{N-1} \sum_{z= 1}^{N-1} \varphi_N (\tfrac{z}{N}) \langle \eta_z \rangle_N $. Observe that the function $\varphi$ is singular at $q=0$ and $q=1$ but that it is integrable on $[0,1]$. Lemma \ref{lem:app} and Remark \ref{rem:refcon} imply that for any $a \in (0,1)$, $\lim_{N \to \infty} |\varphi_N ([Nq]/N) - \varphi (q)| =0$ uniformly in $q \in [a,1-a]$. Therefore we fix some small $a\in (0,1)$ and we split the sum in three sums, one over $z < aN$, one over $aN \le  z \le (1-a)N$ and the last one over $z>(1-a)N$. By using the estimate (\ref{eq:misref}) for $r_N^-$ and similar ones for $r_N^+, {\tilde r}_N^{\pm}$ it is easy to get that 
$$\Big| \varphi_N \big(\tfrac{z}{N}\big) \Big| \le C \Big[ \big(\tfrac{z}{N}\big)^{1-\gamma} +  \big(1- \tfrac{z}{N}\big)^{1-\gamma}\Big]$$ 
so that (use $\langle \eta_z \rangle_N \le 1$)
$$\left|\cfrac{1}{N-1} \sum_{\substack{z<aN\\ z>(1-a)N} } \varphi_N (\tfrac{z}{N}) \langle \eta_z \rangle_N \right| \le C' [ a^{2-\gamma} + (1-a)^{2-\gamma}]$$
for some constants $C,C'>0$ independent of $N$. By using the uniform convergence of $\varphi_N$ to $\varphi$ over $[a,1-a]$ we get that
$$\lim_{N \to \infty} \cfrac{1}{N-1} \sum_{aN \le z \le (1-a)N} \, \varphi_N (\tfrac{z}{N}) \langle \eta_z \rangle_N = \int_{a}^{1-a} \varphi (q) {\bar \rho} (q) dq.$$
Thus sending first $N \to \infty$ and then $a \to 0$ we conclude that
\begin{equation*}
\lim_{N \to \infty} \mu_N \left(\, \langle \pi^N, \varphi_N \rangle \, \right)= \int_0^1 {\bar \rho} (q) {\varphi} (q) dq.
\end{equation*}

Then Theorem \ref{thm:Fick} follows by simple integral computations and using the fact that $\bar \rho$ is the stationary solution of the fractional diffusion equation with Dirichlet boundary conditions.\\

\section*{Acknowledgements}
The authors are very grateful to Patricia Gon\c{c}alves and Claudio Landim for useful discussions and suggestions and specially to Milton Jara for an illuminating discussion on the derivation of Proposition \ref{cor:boundary}. This work has been supported by the Brazilian-French Network in Mathematics and the project EDNHS ANR-14-CE25-0011 of the French National Research Agency (ANR) and the project LSD ANR-15-CE40-0020-01 LSD of the French National Research Agency (ANR).

\appendix

\section{Computations involving the generator}

\begin{lem}
\label{lem:compA}
For any $j \ne k \in \Lambda_N$, we have 
\begin{equation}
\begin{split}
L_N^0 (\eta_j \eta_k) &= \eta_j L_N^0 \eta_k + \eta_k L_N^0 \eta_j -  p(k-j) (\eta_k -\eta_j)^2,\\
L_N^r (\eta_j \eta_k) &= \eta_j L_N^r \eta_k + \eta_k L_N^r \eta_j,\\
L_N^\ell (\eta_j \eta_k) &= \eta_j L_N^\ell \eta_k + \eta_k L_N^\ell \eta_j. 
\end{split}
\end{equation}
\begin{proof}
By definition of $L_N^0$ we have that
\begin{equation*}
\begin{split}
L_N^0 (\eta_j \eta_k) &= \cfrac{1}{2}\sum _{x,y\in \Lambda_{N}}p(y-x)\left[\eta_{j}^{xy}\eta_{k}^{xy}-\eta_{j}\eta_{k}\right]\\
&=\cfrac{1}{2}\sum _{x,y\in \Lambda_{N}}p(y-x)\left[(\eta_{j}^{xy}\eta_{k}-\eta_{j}\eta_{k})+(\eta_{k}^{xy}\eta_{j}-\eta_{j}\eta_{k})+\right. \\
&\;\left. +\eta_{j}^{xy}\eta_{k}^{xy}-\eta_{j}^{xy}\eta_{k}-\eta_{k}^{xy}\eta_{j}+\eta_{j}\eta_{k}\right]\\
&=\eta_j L_N^0 \eta_k + \eta_k L_N^0 \eta_j + \cfrac{1}{2}\sum _{x,y\in \Lambda_{N}}p(y-x)\left[\eta_{j}^{xy}-\eta_{j}\right] \left[ \eta_{k}^{xy}-\eta_{k}\right] \\
&= \eta_j L_N^0 \eta_k + \eta_k L_N^0 \eta_j -  p(k-j) (\eta_k -\eta_j)^2.
\end{split}
\end{equation*}
In order to prove the second expression, note that $\left[\eta_{j}^{x}-\eta_{j}\right] \left[ \eta_{k}^{x}-\eta_{k}\right] = 0$, for all $x \in \bZ$, thus by definition of $L_N^r$ we have
\begin{equation*}
\begin{split}
L_N^r (\eta_j \eta_k)&=\sum_{x\in\Lambda_{N},y\ge N}p(y-x)\left[\eta_{x}(1-\beta)+(1-\eta_{x})\beta\right]\left[(\eta_{j}\eta_{k})^{x}-\eta_{j}\eta_{k}\right]\\
&= \eta_j L_N^r \eta_k + \eta_k L_N^r \eta_j +\\
&+\sum_{x\in\Lambda_{N},y\ge N}p(y-x)\left[\eta_{x}(1-\beta)+(1-\eta_{x})\beta\right]\left[\eta_{j}^{x}-\eta_{j}\right] \left[ \eta_{k}^{x}-\eta_{k}\right]\\
&=\eta_j L_N^r \eta_k + \eta_k L_N^r \eta_j.
\end{split}
\end{equation*}
The proof of the third expression is analogous.
\end{proof}
\end{lem}

\section{Proof of Lemma \ref{lem:app}}
\label{sec:app3}

Let us prove the first item, the second one being similar. It is sufficient to prove it for $q$ in the form $z/N$, $z \ge aN$. We have, by performing an integration by parts, that
\begin{equation*}
\begin{split}
 &N^\gamma r_{N}^{-}(\tfrac{z}{N}) - r^{-}(\tfrac{z}{N}) =N^\gamma \sum_{y \ge z} p(y) - c_{\gamma} \int_{z/N}^{\infty} q^{-\gamma -1} dq\\
&= c_{\gamma} \sum_{y \ge z}  \left[ \tfrac{1}{N}\big( \tfrac{y}{N} \big)^{-\gamma -1} - \int_{y/N}^{(y+1)/N}q^{-\gamma -1} \, dq \right]\\
 &= c_{\gamma} \sum_{y \ge z}  \, \int_{y/N}^{(y+1)/N} \left[ \big( \tfrac{y}{N} \big)^{-\gamma -1}- q^{-\gamma -1} \right]\, dq\\
 &= c_{\gamma} \sum_{y \ge z}  \, \int_{y/N}^{(y+1)/N} \cfrac{d}{dq} \left[ q- \big( \tfrac{y+1}{N} \big)\right] \,\left[ \big( \tfrac{y}{N} \big)^{-\gamma -1}- q^{-\gamma -1} \right]\, dq\\
&=-(\gamma+1) c_\gamma \sum_{y \ge z} \int_{y/N}^{(y+1)/N} q^{-(\gamma +2)} \big(q-\tfrac{y+1}{N} \big) \, dq.
\end{split}
\end{equation*}
Therefore we have that
 \begin{equation}
 \label{eq:misref}
 \Big| N^\gamma r_{N}^{-}(\tfrac{z}{N}) - r^{-}(\tfrac{z}{N})\Big| \le  c_\gamma N^{-1} (z/N)^{-\gamma -1}
 \end{equation}
 which is of order ${\mc O} (N^{-1})$ since $z/N \ge a$. 

For the last claim it is sufficient to prove it for $q=x/N$. By using the symmetry of $p$ we can rewrite
$$({\mc K}_N H)(\tfrac{x}{N})= \cfrac{1}{2} \sum_{z \in \ZZ} p(z) \left[ H(\tfrac{x+z}{N}) +H(\tfrac{x-z}{N}) -2 H(\tfrac{x}{N})\right].$$
We split the sum over $z \in \ZZ$ into a sum over $z \ge 1$ and over $z\le -1$ (recall that $p(0)=0$) and we treat separately the convergence of these two sums. Since the study is the same we consider only the sum over $z\ge 1$. Then, by a discrete integration by parts, we have
\begin{equation*}
\begin{split}
&N^{\gamma} \sum_{z \ge 1} p(z) \left[ H(\tfrac{x+z}{N}) +H(\tfrac{x-z}{N}) -2 H(\tfrac{x}{N})\right]\\
&=\sum_{z=2}^{\infty} N^{\gamma} r_N^- (\tfrac{z}{N}) \left\{ \theta_{\tfrac{x}{N}} (\tfrac{z}{N}) - \theta_{\tfrac{x}{N}} (\tfrac{z-1}{N})\right\} + N^{\gamma} r_N^{-} (\tfrac{1}{N}) \; \theta_{\tfrac{x}{N}} (\tfrac{1}{N})
\end{split}
\end{equation*}
where 
$$\theta_{u} (v)= H(u+v) +H(u-v) -2 H(u).$$
By a second order Taylor expansion of $H$, which is uniform over $x$ since $H$ has compact support, we see that since $\gamma<2$,
$$\lim_{N \to \infty} N^{\gamma} r_N^{-} (\tfrac{1}{N}) \; \theta_{\tfrac{x}{N}} (\tfrac{1}{N})=0$$
uniformly over $x$.
Our aim is now to replace in the remaining sum the term $N^{\gamma} r_N^- (\tfrac{z}{N})$ by $r^{-} (\tfrac{z}{N})$. Recall that we have seen in the proof of the first item that for any $a \in (0,1)$ there exists a constant $C_a>0$ such that
$$| N^{\gamma} r_N^- (\tfrac{z}{N}) - r^{-} (\tfrac{z}{N}) | \le C_a N^{-1}.$$
We rewrite the sum 
$$\sum_{z=2}^{\infty} \left\{ N^{\gamma} r_N^- (\tfrac{z}{N}) -r^{-} (\tfrac{z}{N}) \right\} \;  \left\{ \theta_{\tfrac{x}{N}} (\tfrac{z}{N}) - \theta_{\tfrac{x}{N}} (\tfrac{z-1}{N})\right\} $$
as the sum over $2 \le z \le aN$ and the sum over $z>aN$. In fact the sum over $z>aN$ is equal to the sum over $3N>z>aN$ since for $z\ge 3N$, $ \theta_{\tfrac{x}{N}} (\tfrac{z}{N}) - \theta_{\tfrac{x}{N}} (\tfrac{z-1}{N})=0$. Moreover, we have that $|\theta_{\tfrac{x}{N}} (\tfrac{z}{N}) - \theta_{\tfrac{x}{N}} (\tfrac{z-1}{N})| ={\mc O} (N^{-1})$ uniformly in $x$ and $z$. The sum over $3N>z>aN$ is thus bounded from above by $C^{\prime}_a /N$ for some positive constant $C^{\prime}_a$ (going to $\infty$ as $a$ goes to $0$). Since $\theta_u (v) \le C v^2$ for some positive constant uniformly in $u$, by using the estimate (\ref{eq:misref}) obtained in the proof of the first item, we have also that
\begin{equation*}
\begin{split}
&\left| \sum_{z=2}^{[aN]} \left\{ N^{\gamma} r_N^- (\tfrac{z}{N}) -r^{-} (\tfrac{z}{N}) \right\} \;  \left\{ \theta_{\tfrac{x}{N}} (\tfrac{z}{N}) - \theta_{\tfrac{x}{N}} (\tfrac{z-1}{N})\right\} \right| \\
&\le C^{\prime}  \sum_{z=2}^{[aN]} (\tfrac{z}{N})^2 N^{-1} (z/N)^{-\gamma -1} \le C'' a^{2-\gamma}
\end{split}
\end{equation*}
for constants $C', C''$ which do not depend on $a$ and $x$. In conclusion, the replacement of the term $N^{\gamma} r_N^- (\tfrac{z}{N})$ by $r^{-} (\tfrac{z}{N})$ costs $C'' a^{2-\gamma} + C^{\prime}_a /N$. Therefore, by sending $N \to \infty$ and then $a\to0$, we are reduced to estimate 
\begin{equation*}
\begin{split}
& \sum_{z=2}^{\infty} r^- (\tfrac{z}{N}) \left\{ \theta_{\tfrac{x}{N}} (\tfrac{z}{N}) - \theta_{\tfrac{x}{N}} (\tfrac{z-1}{N})\right\}= \cfrac{1}{N}  \sum_{z=2}^{\infty} r^- (\tfrac{z}{N}) \theta^{\prime}_{\tfrac{x}{N}} (\tfrac{z}{N}) \; + \; \ve_N (x).
\end{split}
\end{equation*}
By a second Taylor expansion, and using that $\gamma<2$, it is easy to see that
$$\lim_{N \to \infty} \sup_{x \in \Lambda_N} |\ve_N (x)|=0.$$
To conclude we observe that there exists $C>0$ such that $|r^{-} (q) \theta^{\prime}_{u} (q) - r^{-} (q') \theta^\prime_{u}(q')| \le C|q-q'| (q\wedge q')^{-\gamma}$, uniformly in $u$. This is because $\theta^{\prime}_u (0) =0$. It follows that for some positive constant $C>0$, we have
\begin{equation*}
\begin{split}
&\left| \cfrac{1}{N}  \sum_{z=2}^{\infty} r^- (\tfrac{z}{N}) \theta^{\prime}_{\tfrac{x}{N}} (\tfrac{z}{N}) - \int_{2/N}^{\infty} r^{-} (q) \theta^{\prime}_{\tfrac{x}{N}} (q) dq \right|\\
& = \left| \sum_{z=2}^{\infty}  \int_{\tfrac{z}{N}}^{\tfrac{z+1}{N}} ( r^- (\tfrac{z}{N}) \theta^{\prime}_{\tfrac{x}{N}} (\tfrac{z}{N}) - r^{-} (q) \theta^{\prime}_{\tfrac{x}{N}} (q))  dq \right|\\
&\le C N^{\gamma -2} \sum_{z=2}^{\infty} z^{-\gamma}
\end{split}
\end{equation*}
where the last term goes to $0$ as $N$ goes to $\infty$.

\section{Proof of Lemma \ref{lem:ghj}}
\label{app:B}
The fact that $\PP^*$ is concentrated  on absolutely continuous measures is obvious since for any continuous function $H:[0,1] \to \RR$ we have
$$|\langle \pi^N , H \rangle| \le \cfrac{1}{(N-1)} \sum_{x=1}^{N-1} | H (x/N)  |$$
and similarly for ${\hat \pi}^N$.
Since for any continuous function $H$, the functional $\pi \in {\mc M}_d^+ \to \langle \pi , H \rangle$ is continuous, by weak convergence, we have that $\PP^*$ is concentrated on measures $(\pi, {\hat \pi})$ such that for any continuous function $H, {\hat H}$
$$|\langle \pi , H \rangle| \le \int_{[0,1]} |H(u)| du, \quad |\langle {\hat \pi} ,\hat H \rangle| \le \int_{[0,1]^2} |{\hat H} (u,v)| du dv$$
which implies that such a $\pi$ and ${\hat \pi}$ are absolutely continuous with respect to the Lebesgue measure. The densities are denoted by $\pi$ and $\hat \pi$. Since ${\hat \pi}^N$ is a product measure whose marginals are given by $\pi^N$, by weak convergence, we have that ${\hat \pi} (u, v) =\pi(u)\pi (v)$ for any $(u,v) \in [0,1]^2$.  

To prove that $\pi$ is continuous we adapt the proof of \cite{KLO} Proposition A.1.1. Let $\nu^{N}_{\rho(\cdot)}$ be the Bernoulli product measure on $\Omega_{N}$ with marginals given by
\begin{equation}
\label{eq:nurho}
\nu_{\rho(\cdot)}^{N}\lbrace \eta_{x} = 1 \rbrace = \rho\left( \dfrac{x}{N}\right),
\end{equation}
where $\rho : [0, 1] \rightarrow [0, 1]$ is a smooth function such that  $\alpha \leq \rho(q)\leq\beta$, for all $q\in [0,1]$, and  $\rho(0) = \alpha$ and $\rho(1)=\beta$. 
 
Let $\ve>0$ be a small real number. Let $F \in C_c^{\infty} ([0,1]^2)$ be a smooth test function and denote by $(\eta (t))_{t \ge 0}$ the boundary driven symmetric long-range exclusion process with generator $N^{\gamma} L_N$. By stationarity of $\mu_N$ and the entropy inequality we have
\begin{equation*}
\begin{split}
&\mu_N \left(  N^{\gamma -1} \sum_{\substack{x,y \in \Lambda_N\\ |x-y| \ge \ve N}} F( \tfrac{x}{N}, \tfrac{y}{N} ) p(y-x) (\eta_y -\eta_x)\right)\\
&={\mathbb E}_{\mu_N} \left( \int_0^1 dt \;  N^{\gamma -1} \sum_{\substack{x,y \in \Lambda_N\\ |x-y| \ge \ve N}}  F( \tfrac{x}{N}, \tfrac{y}{N} ) p(y-x) (\eta_y(t) -\eta_x(t))\right)\\
&\le C_0 + \cfrac{1}{N} \log \left\{ {\mathbb E}_{\nu^{N}_{\rho(\cdot)}} \left( \exp \left[ {N^{\gamma} \int_0^1 dt \, \sum_{\substack{x,y \in \Lambda_N\\ |x-y| \ge \ve N}}  F( \tfrac{x}{N}, \tfrac{y}{N} ) p(y-x) (\eta_y (t) -\eta_x (t) )} \right]\right) \right\}
\end{split}
\end{equation*}
where $C_0$ is a constant resulting from the bound {\footnote{The fact that the relative entropy of $\mu_N$ with respect to $\nu^{N}_{\rho(\cdot)}$ is bounded above by $C_0 N$ with $C_0<\infty$ independent of $N$ can be proved easily since $\{0,1\}$ is compact.}} of the relative entropy of $\mu_N$ with respect to $\nu^{N}_{\rho(\cdot)}$. 

By Feynman-Kac's formula the last expression is bounded by
$$\cfrac{\lambda_N}{N} +C_0$$
where the eigenvalue $\lambda_N$ is given by the variational formula 
\begin{equation}
\label{eq:varfor}
\begin{split}
\lambda_N = \sup_{f} &\left\{ N^{\gamma}  \sum_{\substack{x,y \in \Lambda_N\\ |x-y| \ge \ve N}}  F( \tfrac{x}{N}, \tfrac{y}{N} ) p(y-x) \langle (\eta_y -\eta_x) f(\eta) \rangle_{\nu^{N}_{\rho(\cdot)}} \right.\\
& \left.+  N^{\gamma} \left\langle  L_N   {\sqrt f} , {\sqrt f} \right\rangle_{\nu^{N}_{\rho(\cdot)}}  \right\}
\end{split}
\end{equation}
and the supremum is taken over all the densities $f$ on $\Omega_N$ with respect to $\nu^N_{\rho(\cdot)}$.
Let $F^a$ be the antisymmetric (resp. symmetric) part of $F$, i.e.
$$\forall (u,v) \in [0,1]^2, \quad F^a (u,v) =\cfrac{1}{2} \Big[ F(u,v) -F(v,u) \Big], \quad F^s (u,v) =\cfrac{1}{2} \Big[ F(u,v) + F(v,u) \Big] .$$
Observe that $F^{a} (u,u)=0$ and that $F=F^a +F^s$. We can rewrite
\begin{equation}
\label{eq:pat67}
\begin{split}
& \sum_{\substack{x,y \in \Lambda_N\\ |x-y| \ge \ve N}}  F( \tfrac{x}{N}, \tfrac{y}{N} ) p(y-x) \langle (\eta_y -\eta_x) f(\eta) \rangle_{\nu^{N}_{\rho(\cdot)}} \\
&=  \sum_{\substack{x,y \in \Lambda_N\\ |x-y| \ge \ve N}}  F^a( \tfrac{x}{N}, \tfrac{y}{N} ) p(y-x) \langle (\eta_y -\eta_x) f(\eta) \rangle_{\nu^{N}_{\rho(\cdot)}}\\
\end{split}
\end{equation}
as 
\begin{eqnarray*}
&&\sum_{\substack{x,y \in \Lambda_N\\ |x-y| \ge \ve N}}  F^a( \tfrac{x}{N},\tfrac{y}{N} )  p(y-x) \langle \eta_y \left( f(\eta) - f(\eta^{xy})\right) \rangle_{\nu^{N}_{\rho(\cdot)}}\\
&+&\sum_{\substack{x,y \in \Lambda_N\\ |x-y| \ge \ve N}} F^a( \tfrac{x}{N}, \tfrac{y}{N} ) p(y-x) \langle \eta_y f(\eta^{xy})\left( 1 - \theta^{xy}(\eta)\right)\rangle_{\nu^{N}_{\rho(\cdot)}}\\
&=&\sum_{\substack{x,y \in \Lambda_N\\ |x-y| \ge \ve N}}  F^a( \tfrac{x}{N},\tfrac{y}{N} )  p(y-x) \langle \eta_y \left( f(\eta) - f(\eta^{xy})\right) \rangle_{\nu^{N}_{\rho(\cdot)}}\\
&+&\sum_{\substack{x,y \in \Lambda_N\\ |x-y| \ge \ve N}} F^a( \tfrac{x}{N}, \tfrac{y}{N} ) p(y-x) \langle \eta_x f(\eta)\left(\theta^{xy}(\eta) -1\right)\rangle_{\nu^{N}_{\rho(\cdot)}}\\
&=& (I) + (II)
\end{eqnarray*}
where $\theta^{xy}(\eta)=\tfrac{d\nu_{\rho(\cdot)}^{N}(\eta^{xy})}{d\nu_{\rho(\cdot)}^{N}(\eta)}.$
By  Cauchy-Schwarz inequality, the fact that $f$ is a density and $|\eta_y| \le 1$, we have that $(I)$ is bounded above by
\begin{equation*}
\begin{split}
  \sum_{\substack{x,y \in \Lambda_N\\ |x-y| \ge \ve N}}  \left| F^a ( \tfrac{x}{N},\tfrac{y}{N} ) \right|  p(y-x) \sqrt{ \left\langle [ \sqrt {f  (\eta^{xy})} -\sqrt {f (\eta)}]^2 \right\rangle_{\nu^{N}_{\rho(\cdot)}}}.
 \end{split}
\end{equation*}
Since $\rho(\cdot)$ is Lipshitz we have that $\sup_{\eta \in \Omega_N} \, \left|\theta^{xy}(\eta) -1\right| ={\mc O} ( \tfrac{|x -y|}{N})$. Therefore, by using the elementary inequality $|ab| \le \tfrac{a^2}{2C}+\tfrac{Cb^2}{2}$, and the fact that $f$ is a density, we have that $(II)$ is bounded above by a constant (independent of $N,\ve,F$) times
\begin{equation*}
\begin{split}
& \sum_{\substack{x,y \in \Lambda_N\\ |x-y| \ge \ve N}}   p(y-x) \Big[ F^a \Big( \tfrac{x}{N}, \tfrac{y}{N} \Big)\Big]^2  \; + \;   \sum_{\substack{x,y \in \Lambda_N\\ |x-y| \ge \ve N}}  p(y-x) \left(\cfrac{|x-y|}{N}\right)^2\\
&= c_\gamma N^{1-\gamma} \left\{  \cfrac{1}{N^2}  \sum_{\substack{x,y \in \Lambda_N\\ |x-y| \ge \ve N}}  \cfrac{\Big[ F^a \Big( \tfrac{x}{N}, \tfrac{y}{N} \Big)\Big]^2}{ | \tfrac{x}{N} -\tfrac{y}{N}|^{1+\gamma}}   \; + \;       \cfrac{1}{N^2}  \sum_{\substack{x,y \in \Lambda_N\\ |x-y| \ge \ve N}} | \tfrac{x}{N} -\tfrac{y}{N}|^{1-\gamma}  \right\}.
 \end{split}
\end{equation*}
Observe that
$$\sup_{\ve >0} \sup_{N \ge 1} \cfrac{1}{N^2}  \sum_{\substack{x,y \in \Lambda_N\\ |x-y| \ge \ve N}} | \tfrac{x}{N} -\tfrac{y}{N}|^{1-\gamma} <\infty$$
since $1-\gamma>-1$.

By using (\ref{eq:varfor}), Lemma \ref{lemmabulk}, Cauchy-Schwarz inequality and the previous upper bound for (\ref{eq:pat67}) it follows that there exist constants $C',C'',C''',K$ (independent of $\ve>0$, $N\ge 1$ and $F\in C^\infty_c ([0,1]^2)$) such that
\begin{equation*}
\begin{split}
\cfrac{\lambda_N}{N} & \le N^{\gamma -1} \sup_f \left[ \sum_{\substack{x,y \in \Lambda_N\\ |x-y| \ge \ve N}}  p(y-x) \left(\left| F^a ( \tfrac{x}{N}, \tfrac{y}{N} ) \right|\sqrt{ \left\langle [ \sqrt {f (\eta^{xy})} -\sqrt {f (\eta)}]^2 \right\rangle_{\nu^{N}_{\rho(\cdot)}}} \right.\right.\\
& \left.\left. - C^{\prime} \left\langle [ \sqrt {f (\eta^{xy})} -\sqrt {f (\eta)}]^2 \right\rangle_{\nu^{N}_{\rho(\cdot)}} \right)  \right] \, +  \cfrac{C^{''}}{N^2}  \sum_{\substack{x,y \in \Lambda_N\\ |x-y| \ge \ve N}}  \cfrac{\Big[ F^a \Big( \tfrac{x}{N}, \tfrac{y}{N} \Big)\Big]^2}{   | \tfrac{x}{N} -\tfrac{y}{N}|^{1+\gamma} } \, +K\\
&\le C^{'''} \; \cfrac{1}{N^2} \sum_{x\ne y \in \Lambda_N}  \cfrac{c_\gamma}{| \tfrac{x}{N} -\tfrac{y}{N}|^{1+\gamma}}  \left[ F^a( \tfrac{x}{N},\tfrac{y}{N} ) \right]^2 +K.
\end{split}
\end{equation*}
We have proved that
\begin{equation}
\begin{split}
&\mu^N \left(  N^{\gamma -1} \sum_{\substack{x,y \in \Lambda_N\\ |x-y| \ge \ve N}} F( \tfrac{x}{N}, \tfrac{y}{N} ) p(y-x) (\eta_y -\eta_x)\right)= -2c_\gamma \; \mu^N \left(  \langle \pi^N, g_N \rangle  \right)\\
&\le \cfrac{C^{'''} }{N^2} \sum_{\substack{x,y \in \Lambda_N\\ |x-y| \ge \ve N}}  \cfrac{c_\gamma}{| \tfrac{x}{N} -\tfrac{y}{N}|^{1+\gamma}}  \left[ F^a ( \tfrac{x}{N}, \tfrac{y}{N} ) \right]^2 + K.
\end{split}
\end{equation}
Here $g_N$ is the function defined by
$$\forall u \in [0,1], \quad g_N (u) = \cfrac{1}{N} \sum_{\substack{y \in \Lambda_N\\ \big|\tfrac{y}{N}- u \big| \ge \ve }} \, \cfrac{F^a \big(u, \tfrac{y}{N}\big)}{\vert u -\tfrac{y}{N}\vert^{1+\gamma}}$$
and is a discretization of the smooth function  $g$ defined by
$$ \forall u \in [0,1], \quad g(u) = \int_{\substack{y\in [0,1], \\ |y-u| \ge \ve}} \cfrac{F^a (u,y)}{|y-u|^{1+\gamma}} \, dy.$$
Let $Q_{\ve}=\{(u,v)\in [0,1]^2\; ; \; |u-v| \ge \ve\}$.Observe first that for symmetry reasons we have that for any integrable function $\pi$, 
$$\int_0^1 \pi (u) g(u) du = \cfrac{1}{2} \iint_{Q_\ve}\cfrac{(\pi (v) -\pi (u)) F^{a} (u,v) }{|u-v|^{1+\gamma}}\; du dv.$$
We take the limit $N \to \infty$. We conclude that there exist constants $C,C'>0$ independent of $F \in C_c^{\infty} ([0,1]^2)$ and $\ve>0$ such that
\begin{equation}
{\bb E}^* \left[  \iint_{Q_\ve}\cfrac{(\pi (v) -\pi (u)) F^{a} (u,v) }{|u-v|^{1+\gamma}}\; du dv \; -\;   C \iint_{Q_{\ve}}\cfrac{ \big[ F^{a} (u,v) \big]^2 }{|u-v|^{1+\gamma}} \; du dv  \right] \le C'.
\end{equation}
It is easy to see that the supremum over $F$ can be inserted in the expectation (see Lemma 7.5 in \cite{KLO2}) so that
\begin{equation}
{\bb E}^* \left[ \sup_F \left\{   \iint_{Q_{\ve} }\cfrac{(\pi (v) -\pi (u)) F^{a} (u,v) }{|u-v|^{1+\gamma}}\; du dv \; -\;   C \iint_{Q_{\ve}}\cfrac{ \big[ F^{a} (u,v) \big]^2 }{|u-v|^{1+\gamma}} \; du dv     \right\} \right] \le C'.
\end{equation}
By writing $F=F^a + F^s$, and observing that the function $(u,v) \in [0,1]^2 \to \pi(v) -\pi (u)$ is antisymmetric,  we have that 
$$ \iint_{Q_{\ve} }\cfrac{(\pi (v) -\pi (u)) F^{a} (u,v) }{|u-v|^{1+\gamma}}\; du dv \, =\,  \iint_{Q_{\ve} }\cfrac{(\pi (v) -\pi (u)) F (u,v) }{|u-v|^{1+\gamma}}\; du dv.$$
Moreover, by using the definition of $F^a$ and using the inequality $(\tfrac{a+b}{2})^2 \le \tfrac{a^2 + b^2}{2}$, it is easy to see that
$$\iint_{Q_{\ve}}\cfrac{ \big[ F^a (u,v) \big]^2 }{|u-v|^{1+\gamma}} \; du dv \, \le \, \iint_{Q_{\ve}}\cfrac{ \big[ F (u,v) \big]^2 }{|u-v|^{1+\gamma}} \; du dv.$$
It follows that
\begin{equation}
\label{eq:A3}
{\bb E}^* \left[ \sup_F \left\{   \iint_{Q_{\ve} }\cfrac{(\pi (v) -\pi (u)) F (u,v) }{|u-v|^{1+\gamma}}\; du dv \; -\;   C \iint_{Q_{\ve}}\cfrac{ \big[ F (u,v) \big]^2 }{|u-v|^{1+\gamma}} \; du dv     \right\} \right] \le C'.
\end{equation}

Consider the Hilbert space ${\mathbb L}^2 ([0,1]^2, d\mu_{\ve})$ where $\mu_{\ve}$ is the measure whose density with respect to Lebesgue measure is
$$ (u,v) \in [0,1]^2 \to {\bf 1}_{|u-v| \ge \ve} \, |u-v|^{-(1+\gamma)}.$$
By letting 
$$\Pi: (u,v) \in [0,1]^2 \to \pi(v) -\pi (u)$$ 
the previous formula implies that 
$$\EE^* \left[ \iint_{[0,1]^2} \Pi^2 (u,v) \, d\mu_{\ve} (u,v) \right] \le 4 CC'.$$
Letting $\ve \to 0$, by the monotone convergence theorem, we conclude that
\begin{equation*}
 \iint_{[0,1]^2}\cfrac{(\pi (v) -\pi (u))^2}{|u-v|^{1+\gamma}}\; du dv
\end{equation*}
is finite $\PP^*$ a.s.. It follows from Theorem 8.2 of  \cite{dPV} that $\PP^*$ almost surely $\pi$ is $\tfrac{\gamma-1}{2}$-H\"older. This concludes the proof of Lemma \ref{lem:ghj}.

\begin{lem}\label{lemmabulk}
Let $f$ be a density with respect to the product measure $\nu_{\rho(\cdot)}^{N}$ defined by (\ref{eq:nurho}). Then, there exist constants $C_{\alpha,\beta}, C_{\alpha,\beta}^{\prime}$ such that 
$$\langle L_{N}\sqrt{f},\sqrt{f} \rangle_{\nu^{N}_{\rho(\cdot)}}  \le  -C_{\alpha,\beta}^{\prime} D_{N} (f) + \dfrac{C_{\alpha,\beta}}{N^{\gamma -1}}\le -C_{\alpha,\beta}^{\prime} D^0_{N} (f) + \dfrac{C_{\alpha,\beta}}{N^{\gamma -1}}$$
where $D_N (f) =D_N^0 (f) +D_N^{\ell} (f)+D_N^r (f)$ with
\begin{equation*}
\begin{split}
&{ D}^0_{N} (f) = \sum_{x,y \in \Lambda_N} p(y-x) \left\langle [ \sqrt {f (\eta^{xy})} -\sqrt {f (\eta)}]^2 \right\rangle_{\nu_{\rho(\cdot)}^{N}},\\
&{ D}^\ell_{N} (f) = \sum_{x \in \Lambda_N, y \le 0} p(y-x) [\eta_x (1-\alpha) +\alpha (1-\eta_x)] \, \left\langle [ \sqrt {f (\eta^{x})} -\sqrt {f (\eta)}]^2 \right\rangle_{\nu_{\rho(\cdot)}^{N}},\\
&{ D}^r_{N} (f) = \sum_{x \in \Lambda_N, y \ge N} p(y-x) [\eta_x (1-\beta) +\beta (1-\eta_x)] \, \left\langle [ \sqrt {f (\eta^{x})} -\sqrt {f (\eta)]}^2 \right\rangle_{\nu_{\rho(\cdot)}^{N}}.
\end{split}
\end{equation*}
\end{lem}

\begin{proof}
In the proof $C$ and $C^\prime$ are constants depending on $\alpha$ and $\beta$ and $\rho(\cdot)$ whose value can change from line to line. We are going to show that
\begin{equation}
\label{eq:byron}
\begin{split}
&\langle L_{N}^{0}\sqrt{f},\sqrt{f} \rangle_{\nu^{N}_{\rho(\cdot)}}  \le  - C^{\prime} D_{N}^{0}(f) + \dfrac{C}{N^{\gamma -1}},\\
&\langle L_{N}^{\ell}\sqrt{f},\sqrt{f} \rangle_{\nu^{N}_{\rho(\cdot)}}  \le  - C^{\prime} D_{N}^{\ell}(f) + \dfrac{C}{N^{\gamma -1}},\\
&\langle L_{N}^{r}\sqrt{f},\sqrt{f} \rangle_{\nu^{N}_{\rho(\cdot)}}  \le  - C^{\prime} D_{N}^{r}(f) + \dfrac{C}{N^{\gamma -1}}.
\end{split}
\end{equation}
We have that 
$$\displaystyle \langle L_{N}^{0}\sqrt{f},\sqrt{f} \rangle_{\nu_{\rho(\cdot)}^{N}} =\dfrac{1}{2} \sum _{x,y \in \Lambda_{N}} p(x-y)\langle L_{x,y}^{0}\sqrt{f},\sqrt{f} \rangle_{\nu_{\rho(\cdot)}^{N}}$$
where $\displaystyle\langle L_{x,y}^{0}\sqrt{f},\sqrt{f} \rangle_{\nu_{\rho(\cdot)}^{N}} = \int p(x-y) \left[ \sqrt{f(\eta^{xy})}-\sqrt{f(\eta)}\right]\sqrt{f(\eta)}d\nu_{\rho(\cdot)}^{N}(\eta)$. Thus, recalling that $\theta^{xy}(\eta)=\tfrac{d\nu_{\rho(\cdot)}^{N}(\eta^{xy})}{d\nu_{\rho(\cdot)}^{N}(\eta)}$ we obtain the following
\begin{eqnarray*}
\displaystyle\langle L_{x,y}^{0}\sqrt{f},\sqrt{f} \rangle_{\nu_{\rho(\cdot)^{N}}}
&=&\dfrac{1}{2}\int p(y-x) \left[ \sqrt{f(\eta^{xy})}-\sqrt{f(\eta)}\right]\sqrt{f(\eta)}d\nu_{\rho(\cdot)}^{N}(\eta)\\
&-&\dfrac{1}{2}\int p(y-x) \left[ \sqrt{f(\eta^{xy})}-\sqrt{f(\eta)}\right]\sqrt{f(\eta^{xy})}d\nu_{\rho(\cdot)}^{N}(\eta)\\
&+& \dfrac{1}{2} \int p(y-x) \left[ \sqrt{f(\eta^{xy})}\right]^{2} \left[1-\theta^{xy}(\eta)\right] d\nu_{\rho(\cdot)}^{N}(\eta)\\
&=&-\dfrac{1}{2} \int p(y-x)  \left[ \sqrt{f(\eta^{xy})}-\sqrt{f(\eta)}\right]^{2}d\nu_{\rho(\cdot)}^{N}(\eta)\\
&+&\dfrac{1}{2}\int p(y-x) \left[ \sqrt{f(\eta^{xy})}\right]^{2} \left[1-\theta^{xy}(\eta)\right] d\nu_{\rho(\cdot)}^{N}(\eta).
\end{eqnarray*}
Thus we have that
\begin{eqnarray*}
\displaystyle \langle L_{N}^{0}\sqrt{f},\sqrt{f} \rangle_{\nu_{\rho(\cdot)}^{N}}&\leq& -\dfrac{1}{4}D_{N}^{0}(f)\\
&&+\dfrac{1}{4}\sum _{x,y \in \Lambda_{N}} p(x-y)\int \left[ \sqrt{f(\eta^{xy})}\right]^{2} \left[1-\theta^{xy}(\eta)\right] d\nu_{\rho(\cdot)}^{N}(\eta). 
\end{eqnarray*}
The second term on the right hand side of the last expression is equal to 
\begin{eqnarray*}
&\dfrac{1}{8}\sum _{x,y \in \Lambda_{N}} \int p(x-y) \left[ \sqrt{f(\eta^{xy})}\right]^{2} \left[1-\theta^{xy}(\eta)\right] d\nu_{\rho(\cdot)}^{N}(\eta)\\
&+\dfrac{1}{8}\sum _{x,y \in \Lambda_{N}} \int p(x-y) \left[ \sqrt{f(\eta)}\right]^{2} \left[\theta^{xy}(\eta)-1\right] d\nu_{\rho(\cdot)}^{N}(\eta),
\end{eqnarray*}
so using that for any $\ve>0$, $ab\leq \tfrac{\ve^2 a^{2}}{2} + \tfrac{b^{2}}{2\ve^2}$ and writing $[\sqrt{f (\eta^{xy})}]^2 - [\sqrt {f (\eta)}]^2 = [\sqrt {f (\eta^{xy}}) -\sqrt {f (\eta)} ][\sqrt {f (\eta^{xy})} +\sqrt {f (\eta)}]$ the last expression is absolutely bounded by
\begin{eqnarray*}
&&C \ve^2  \sum _{x,y \in \Lambda_{N}} p(x-y)\int \left[ \sqrt{f(\eta^{xy})}- \sqrt{f(\eta)}\right]^{2} d\nu_{\rho(\cdot)}^{N}(\eta)\\
&+&{C} \ve^{-2} \sum _{x,y \in \Lambda_{N}} p(x-y)\int \left[ \sqrt{f(\eta)}+ \sqrt{f(\eta^{xy})}\right]^{2} \left[1-\theta^{xy}(\eta)\right]^{2} d\nu_{\rho(\cdot)}^{N}(\eta)\\
&\leq & C \ve^2 D_{N}^{0}(f)+ C \ve^{-2} N^{-2}\sum _{y\ne x \in \Lambda_{N}}\dfrac{1}{|x-y|^{\gamma -1}},\\
&\leq & C\ve^2  D_{N}^{0} (f)+ C\ve^{-2}  N^{-\gamma +1}.
\end{eqnarray*}
We choose then $\ve>0$ sufficiently small to have $C'=1/4 - C\ve^2>0$. Then the first inequality in (\ref{eq:byron}) follows.

Now we prove only the second inequality in (\ref{eq:byron}) since the third one can be proved similarly. We have that 
$$\langle L_{N}^{\ell}\sqrt{f},\sqrt{f} \rangle_{\nu_{\rho(\cdot)}^{N}} = \sum _{\substack{y\leq 0,\\x \in \Lambda_{N}}} p(x-y)\langle L_{x}^{\ell}\sqrt{f},\sqrt{f} \rangle_{\nu_{\rho(\cdot)}^{N}}$$
where $\langle L_{x}^{\ell}\sqrt{f},\sqrt{f} \rangle_{\nu_{\rho(\cdot)}^{N}} = \int I_{\alpha}^{x}(\eta) \left[ \sqrt{f(\eta^{x})}-\sqrt{f(\eta)}\right]\sqrt{f(\eta)}d\nu_{\rho(\cdot)}^{N}(\eta)$ and $I_{\alpha}^x= [ \eta_x (1-\alpha) + (1-\eta_x) \alpha] $.
Thus, denoting $\theta^{x}(\eta)=\dfrac{d\nu^N_{\rho(\cdot)}(\eta^{x})}{d\nu_{\rho(\cdot)}^{N}(\eta)}$ we obtain the following
\begin{eqnarray*}
\displaystyle\langle L_{x}^{\ell}\sqrt{f},\sqrt{f} \rangle_{\nu_{\rho(\cdot)}^{N}}
&=&-\dfrac{1}{2}\int I_{\alpha}^{x}(\eta) \left[ \sqrt{f(\eta^{x})}-\sqrt{f(\eta)}\right]^{2}d\nu_{\rho(\cdot)}^{N}(\eta)\\
&+&\dfrac{1}{2}\int \left[ \sqrt{f(\eta^{x})}\right]^{2} \left[I_{\alpha}^{x}(\eta)-I_{\alpha}^{x}(\eta^{x})\theta^{x}(\eta)\right] d\nu_{\rho(\cdot)}^{N}(\eta).\\
\end{eqnarray*}
Performing a change of variables we have that the second term on the right hand side of the last expression can be written as
\begin{eqnarray*}
&&\dfrac{1}{4}\int \left[ \sqrt{f(\eta^{x})}\right]^{2} \left[I_{\alpha}^{x}(\eta)-I_{\alpha}^{x}(\eta^{x})\theta^{x}(\eta)\right] d\nu_{\rho(\cdot)}^{N}(\eta)\\
&-&\dfrac{1}{4}\int \left[ \sqrt{f(\eta)}\right]^{2} \left[I_{\alpha}^{x}(\eta)-I_{\alpha}^{x}(\eta^{x})\theta^{x}(\eta)\right] d\nu_{\rho(\cdot)}^{N}(\eta)\\
&=&\dfrac{1}{4}\int \left(\left[ \sqrt{f(\eta^{x})}\right]^{2}-\left[ \sqrt{f(\eta)}\right]^{2}\right) \left[I_{\alpha}^{x}(\eta)-I_{\alpha}^{x}(\eta^{x})\theta^{x}(\eta^{x})\right] d\nu_{\rho(\cdot)}^{N}(\eta).
\end{eqnarray*}
Using again the inequality $ab\le \tfrac{\ve^2 a^{2}}{2} + \tfrac{b^{2}}{2\ve^2}$, $\ve>0$, the integral above is absolutely bounded by
\begin{eqnarray*}
&& C\ve^2 \int I_{\alpha}^{x}(\eta)\left(\left[ \sqrt{f(\eta^{x})}\right]-\left[ \sqrt{f(\eta)}\right]\right)^{2}  d\nu_{\rho(\cdot)}^{N}(\eta)\\
&+& C\ve^{-2} \int \dfrac{1}{I_{\alpha}^{x}(\eta)}\left[I_{\alpha}^{x}(\eta)-I_{\alpha}^{x}(\eta^{x})\theta^{x}(\eta)\right]^{2}\left(\left[ \sqrt{f(\eta^{x})}\right]+\left[ \sqrt{f(\eta)}\right]\right)^{2}  d\nu_{\rho(\cdot)}^{N}(\eta)\\
&\leq & C\ve^2 \int I_{\alpha}^{x}(\eta)\left(\left[ \sqrt{f(\eta^{x})}\right]-\left[ \sqrt{f(\eta)}\right]\right)^{2}  d\nu_{\rho(\cdot)}^{N}(\eta)\\
&+& 2 C\ve^{-2} \int \dfrac{1}{I_{\alpha}^{x}(\eta)}\left[I_{\alpha}^{x}(\eta)-I_{\alpha}^{x}(\eta^{x})\theta^{x}(\eta)\right]^{2}\left( f(\eta^{x})+ f(\eta)\right)  d\nu_{\rho(\cdot)}^{N}(\eta).\\
\end{eqnarray*}
Now, by the smoothness of $\rho$ and the fact that $\rho (0)=\alpha$ we have that 
$$\dfrac{1}{I_{\alpha}^{x}(\eta)}\left[I_{\alpha}^{x}(\eta)-I_{\alpha}^{x}(\eta^{x})\theta^{x}(\eta)\right]^{2} \leq C \dfrac{x^{2}}{N^{2}}.$$
Thus, by using the fact that $f$ is a density and that $\theta_x$ is uniformly bounded, and by choosing $\ve>0$ sufficiently small, we get 
\begin{eqnarray*}
\langle L_{N}^{\ell}\sqrt{f},\sqrt{f} \rangle_{\nu_{\rho(\cdot)}} &\leq&-C^\prime  D_{N}^{\ell} (f)+\dfrac{C}{N^{2}}\sum_{\substack{y\leq 0,\\x \in \Lambda_{N}}}\dfrac{x^{2}}{[x-y]^{\gamma + 1}}\\
& \le &-C^\prime  D_{N}^{\ell}(f) + \mathcal{O}(N^{1-\gamma})
\end{eqnarray*}
which proves the second inequality in (\ref{eq:byron}).
\end{proof}


\end{document}